\documentclass[12pt,a4paper]{article}
\usepackage{amssymb}
\usepackage{amsmath, amscd}
\usepackage{amsthm}
\usepackage{mathtools}
\usepackage{mathrsfs}
\usepackage{booktabs}
\usepackage{perpage}
\usepackage[all,cmtip]{xy}
\usepackage{setspace}
\usepackage{amsbsy}
\usepackage[T1]{fontenc}
\usepackage{verbatim}
\usepackage{mathdots}
\usepackage{mathrsfs}
\usepackage{ifthen}
\usepackage{blkarray}
\usepackage{multirow}
\usepackage{enumerate}
\usepackage{colonequals}
\usepackage{braket}

\usepackage[dvipsnames]{xcolor}

\usepackage{lipsum}

\mathtoolsset{showonlyrefs=true}

\numberwithin{equation}{subsection}

\newtheorem{theorem}{Theorem}[subsection]
\newtheorem{lemma}[theorem]{Lemma}

\newtheorem{example}[theorem]{Example}
\newtheorem{corollary}[theorem]{Corollary}
\newtheorem{definition}[theorem]{Definition}

\newtheorem{proposition}[theorem]{Proposition}

\newtheorem*{thm1}{Theorem 1}
\newtheorem*{thm2}{Theorem 2}

\newtheorem*{cor3}{Corollary 3}

\theoremstyle{remark}

\newtheorem{rmk}[theorem]{Remark}

\newcommand{\GZip}{\mathop{\text{$G$-{\tt Zip}}}\nolimits}

\newcommand{\GF}{\mathop{\text{$G$-{\tt ZipFlag}}}\nolimits}

\newcommand{\VB}{\mathfrak{VB}}

\newskip\procskipamount
\newskip\interskipamount
\newskip\refskipamount
\newcommand{\procskip}{\vskip\procskipamount}
\newcommand{\interskip}{\vskip\interskipamount}
\newcommand{\refskip}{\vskip\refskipamount}

\newcommand{\procbreak}{\par
   \ifdim\lastskip<\procskipamount\removelastskip
   \penalty-100
   \procskip\fi
   \noindent\ignorespaces}

\newcommand{\titlebreak}{\par%
\ifdim\lastskip<\interskipamount\removelastskip%
\penalty10000%
\interskip\fi%
\noindent}%

\newcommand{\interbreak}{\par%
\ifdim\lastskip<\interskipamount\removelastskip%
\penalty-100%
\interskip\fi%
\noindent\ignorespaces}%

\newcommand{\refbreak}{\par%
\ifdim\lastskip<\refskipamount\removelastskip%
\penalty-100%
\refskip\fi%
\noindent\ignorespaces}%

\newcounter{listcounter}
\newcounter{deflistcounter}
\newcounter{equivcounter}

\newskip{\itemsepamount}
\newskip{\topsepamount}
\newenvironment{assertionlist}{%
  \begin{list}
    {\upshape (\arabic{listcounter})}
    {\setlength{\leftmargin}{18pt}
     \setlength{\rightmargin}{0pt}
     \setlength{\itemindent}{0pt}
     \setlength{\labelsep}{5pt}
     \setlength{\labelwidth}{13pt}
     \setlength{\listparindent}{\parindent}
     \setlength{\parsep}{0pt}
     \setlength{\itemsep}{\itemsepamount}
     \setlength{\topsep}{\topsepamount}
     \usecounter{listcounter}}}
  {\end{list}}

\newenvironment{definitionlist}{%
  \begin{list}
    {\upshape (\alph{deflistcounter})}
    {\setlength{\leftmargin}{18pt}
     \setlength{\rightmargin}{0pt}
     \setlength{\itemindent}{0pt}
     \setlength{\labelsep}{5pt}
     \setlength{\labelwidth}{13pt}
     \setlength{\listparindent}{\parindent}
     \setlength{\parsep}{0pt}
     \setlength{\itemsep}{\itemsepamount}
     \setlength{\topsep}{\topsepamount}
     \usecounter{deflistcounter}}}
  {\end{list}}

\newenvironment{equivlist}{%
  \begin{list}
    {\upshape (\roman{equivcounter})}
    {\setlength{\leftmargin}{18pt}
     \setlength{\rightmargin}{0pt}
     \setlength{\itemindent}{0pt}
     \setlength{\labelsep}{5pt}
     \setlength{\labelwidth}{13pt}
     \setlength{\listparindent}{\parindent}
     \setlength{\parsep}{0pt}
     \setlength{\itemsep}{\itemsepamount}
     \setlength{\topsep}{\topsepamount}
     \usecounter{equivcounter}}}
  {\end{list}}

% ...

\newcommand{\Bcal}{{\mathcal B}}

\newcommand{\Fcal}{{\mathcal F}}

\newcommand{\Hcal}{{\mathcal H}}

\newcommand{\Lcal}{{\mathcal L}}

\newcommand{\Ncal}{{\mathcal N}}
\newcommand{\Ocal}{{\mathcal O}}

\newcommand{\Ucal}{{\mathcal U}}
\newcommand{\Vcal}{{\mathcal V}}

\newcommand{\Xcal}{{\mathcal X}}

\newcommand{\Zcal}{{\mathcal Z}}

\renewcommand{\AA}{\mathbb{A}}

\newcommand{\FF}{\mathbb{F}}
\newcommand{\GG}{\mathbb{G}}

\newcommand{\NN}{\mathbb{N}}

\newcommand{\PP}{\mathbb{P}}
\newcommand{\QQ}{\mathbb{Q}}
\newcommand{\RR}{\mathbb{R}}

\newcommand{\ZZ}{\mathbb{Z}}

\newcommand{\Ascr}{{\mathscr A}}

\newcommand{\Escr}{{\mathscr E}}

\newcommand{\Sscr}{{\mathscr S}}

\newcommand{\Vscr}{{\mathscr V}}

\DeclareMathOperator{\ad}{ad}

\newcommand{\dR}{{\rm dR}}

\DeclareMathOperator{\Gal}{Gal}

\DeclareMathOperator{\hdg}{Hdg}
\DeclareMathOperator{\Hom}{Hom}

\DeclareMathOperator{\Span}{Span}

\DeclareMathOperator{\pr}{pr}

\DeclareMathOperator{\Ker}{Ker}

\DeclareMathOperator{\Rep}{Rep}
\DeclareMathOperator{\res}{Res}

\DeclareMathOperator{\Sbt}{Sbt}
\DeclareMathOperator{\Sh}{Sh}

\DeclareMathOperator{\Sym}{Sym}

\DeclareMathOperator{\zip}{zip}

\DeclareMathOperator{\loc}{loc}
\DeclareMathOperator{\Sht}{Sht}

\DeclareMathOperator{\SL}{SL}
\DeclareMathOperator{\GL}{GL}

\DeclareMathOperator{\GSp}{GSp}
\DeclareMathOperator{\Sp}{Sp}
\DeclareMathOperator{\U}{U}
\DeclareMathOperator{\GU}{GU}

\newcommand{\shgx}{\Sh(\mathbf G, \mathbf X)}

\newcommand{\egx}{E(\mathbf G, \mathbf X)}
\newcommand{\gx}{(\mathbf G, \mathbf X)}

\newcommand{\gofaf}{\mathbf G(\mathbf A_f)}

\newcommand{\id}{{\rm Id}}

\newcommand{\loccit}{{\em loc.\ cit. }}
\newcommand{\loccitn}{{\em loc.\ cit.}}
\newcommand{\cf}{{\em cf. }}

\newcommand{\diag}{{\rm diag}}

\newcommand{\fil}{{\rm Fil}}

\renewcommand{\Im}{{\rm Im}}

\newcommand{\xg}{\mathbf X_{g}}

\newcommand{\shdagsp}{(\GSp(2g), \xg)}

\DeclareMathOperator{\Std}{Std}

\DeclareMathOperator{\conj}{conj}

\DeclareMathOperator{\Fil}{Fil}
\DeclareMathOperator{\Isom}{Isom}

\DeclareMathOperator{\pf}{pf}

\DeclareMathOperator{\Ind}{Ind}

\DeclareMathOperator{\vs}{Vec}
\DeclareMathOperator{\deter}{det}

\newcommand{\relmiddle}[1]{\mathrel{}\middle#1\mathrel{}}

\begin{document}

\author{Naoki Imai and Jean-Stefan Koskivirta}

\title{Automorphic vector bundles on the stack of $G$-zips} 

\date{}

\maketitle
\begin{abstract}
For a connected reductive group $G$ over a finite field, we study automorphic vector bundles on the stack of $G$-zips. In particular, we give a formula in the general case for the space of global sections of an automorphic vector bundle in terms of the Brylinski--Kostant filtration. 
Moreover, we give an equivalence of categories between the category of automorphic vector bundles on the stack of $G$-zips and a category of admissible  modules with actions of a zero-dimensional algebraic subgroup a Levi subgroup and monodromy operators. 
\end{abstract}


\section{Introduction}
The stack of $G$-zips was introduced by Pink--Wedhorn--Ziegler (\cite{Pink-Wedhorn-Ziegler-zip-data} and \cite{PinkWedhornZiegler-F-Zips-additional-structure}) based on the notion of F-zip defined in the work of Moonen--Wedhorn (\cite{Moonen-Wedhorn-Discrete-Invariants}). 
In this paper, we investigate vector bundles on the stack of $G$-zips. Let $G$ be a connected reductive group over a finite field $\FF_q$ and let $k$ denote an algebraic closure of $\FF_q$. For a cocharacter $\mu \colon \GG_{\mathrm{m},k}\to G_k$, Pink--Wedhorn--Ziegler have defined a smooth finite stack $\GZip^\mu$ over $k$, called the stack of $G$-zips of type $\mu$. Many authors have shown that it is a useful tool to study the geometry of Shimura varieties in characteristic $p$. 
For example, let $\shgx_{K}$ be a Shimura variety of Hodge-type over a number field $\mathbf{E}$ with good reduction at a prime $p$. Kisin  (\cite{Kisin-Hodge-Type-Shimura}) and Vasiu (\cite{Vasiu-Preabelian-integral-canonical-models}) have constructed an integral model $\Sscr_K$ over $\Ocal_{\mathbf{E}_v}$ at all places $v|p$ in $\mathbf{E}$. Denote by $S_K$ the geometric special fiber of $\Sscr_K$ and by $G$ the special fiber over $\FF_p$ of $\mathbf{G}$ (in the context of Shimura varieties, we take $q=p$). 
Let $\mu$ be the cocharacter attached naturally to $\mathbf{X}$. Then Zhang (\cite{Zhang-EO-Hodge}) has shown that there exists a smooth morphism of stacks $\zeta \colon S_K\to \GZip^\mu$, which is also surjective. The second author and Wedhorn have used the stack $\GZip^\mu$ to construct $\mu$-ordinary Hasse invariants in \cite{Koskivirta-Wedhorn-Hasse}, and this result was later generalized to all Ekedahl--Oort strata with Goldring (\cite{Goldring-Koskivirta-Strata-Hasse}).

In the paper \cite{Koskivirta-automforms-GZip}, the second author studied the space of global sections of the family of vector bundles $(\mathcal{V}_I (\lambda))_{\lambda \in X^*(T)}$. To explain what these vector bundles are, first recall that the cocharacter $\mu$ yields a parabolic subgroup $P\subset G_k$ as well as a Levi subgroup $L\subset P$, which is equal to the centralizer of $\mu$ (see \S \ref{subsec-cochar} for details). 
Then for any algebraic $P$-representation $(V,\rho)$ over $k$, there is a naturally attached vector bundle $\mathcal{V}(\rho)$ of rank $\dim (V)$ on $\GZip^\mu$ modeled on $(V,\rho)$ (see \S \ref{sec-vector-bundles-gzipz}). 
We call $\Vcal(\rho)$ an automorphic vector bundle on $\GZip^\mu$ (\cf \cite[III. 2]{Milne-ann-arbor}). 

The vector bundle $\mathcal{V}_I (\lambda)$ (for $\lambda\in X^*(T)$ a character of a maximal torus $T\subset G$) is by definition the vector bundle attached to the $P$-representation $V_I (\lambda)=\Ind_B^P(\lambda)$, where $B\subset P$ is a Borel subgroup (containing $T$, and appropriately chosen), $\Ind$ denotes induction  and $I$ denotes the set of simple roots of $L$. 
For a $k$-algebraic group $H$, we write $\Rep(H)$ for the category of finite dimensional algebraic representations of $H$ over $k$. The natural projection $P\to L$ modulo the unipotent radical induces a fully faithful functor $\Rep(L)\to \Rep(P)$. In particular, all representations of the form $V_I (\lambda)$ lie in the full subcategory $\Rep(L)$. In the case when $G$ is split over $\FF_p$, we showed in a previous work (\cite[Theorem 1]{Koskivirta-automforms-GZip}) that $H^0(\GZip^\mu,\mathcal{V}_I (\lambda))$ can be expressed as 
\begin{equation}\label{intor-formula-Vlambda}
H^0(\GZip^\mu,\mathcal{V}_I (\lambda))=V_I (\lambda)^{L(\FF_p)} \cap V_I (\lambda)_{\leq 0}
\end{equation}
where $V_I (\lambda)^{L(\FF_p)}$ denotes the $L(\FF_p)$-invariant subspace of $V_I (\lambda)$ and $V_I (\lambda)_{\leq 0}\subset V_I (\lambda)$ is defined as follows: It is the direct sum of the $T$-weight spaces $V_I (\lambda)_\nu$ for the weights $\nu$ satisfying $\langle \nu, \alpha^\vee \rangle \leq 0$ for any simple root $\alpha$ outside of $L$.

In this paper, we vastly generalize the formula \eqref{intor-formula-Vlambda} to the most general case. We do not assume that $G$ is split over $\FF_q$, and more importantly, we consider arbitrary representations in the larger category $\Rep(P)$ as opposed to the subcategory $\Rep(L)$. In the context of Shimura varieties, there are many interesting vector bundles other than the family $(V_I (\lambda))_\lambda$, which may not always arise from representations in $\Rep(L)$. 
For example, in the article \cite{Urban-nearly-holomorphic}, nearly-holomorphic modular forms of weight $k$ and order $\leq r$ are defined as sections of the vector bundle $\omega^{\otimes (k-r)}\Sym^r(\Hcal^1_{\dR})$ on the modular curve $X(N)$ for some level $N\geq 1$. Here, $\Hcal^1_{\dR}$ is the sheaf of relative de Rham cohomology of the universal elliptic curve $\Escr\to X(N)$, and $0\subset \omega \subset \Hcal^1_{\dR}$ is the usual Hodge filtration. In this context, the group $G$ is $\GL_{2}$, $P=B$ is a Borel subgroup of $G$. 
The vector bundle $\Hcal^1_{\dR}$ is attached to the dual of the standard representation of $\GL_2$ (viewed by restriction as a representation of $P$). Similarly, $\Sym^r(\Hcal^1_{\dR})$ is attached to the $r$-th symmetric power of that representation. More generally, on the Siegel-type Shimura variety $\Ascr_g$ (which parametrize principally polarized abelian varieties of rank $g$), the universal abelian scheme yields a rank $2g$ vector bundle $\Hcal^1_{\dR}$ on $\Ascr_{g}$. One can extend the definition of $\Hcal^1_{\dR}$ to Hodge-type Shimura varieties after choosing a Siegel embedding. Furthermore, it extends to a vector bundle on the integral model $\Sscr_K$ of Kisin and Vasiu. 
This example shows that it is desirable to also understand vector bundles that arise from general representations of $P$. In this paper, we determine the space $H^0(\GZip^\mu,\mathcal{V}(\rho))$ for any cocharacter datum $(G,\mu)$ (for the definition of cocharacter datum, see \S \ref{subsec-cochar}) and for any representation $(V,\rho)\in \Rep(P)$. By Zhang's smooth surjective map $\zeta:S_K\to \GZip^\mu$, this determines a natural Hecke-equivariant subspace
\begin{equation}\label{inclusion-intro}
\xymatrix@1@M=5pt{
H^0(\GZip^\mu,\mathcal{V}(\rho)) \ar@{^{(}->}[r]_-{\zeta^*} & 
H^0(S_K,\mathcal{V}(\rho)).
}    
\end{equation}
In particular, we obtain Hecke-equivariant sections of $\mathcal{V}(\rho)$ on $S_K$. Furthermore, we can potentially study sections on Ekedahl--Oort strata by the same method, as demonstrated in \cite{Goldring-Koskivirta-Strata-Hasse}. Another motivation for describing sections on $\GZip^\mu$ is that we would like to determine which weights $\lambda$ admit nonzero automorphic forms. Specifically, let $C_{K}$ denote the set of $\lambda\in X^*(T)$ such that $H^0(S_K,\mathcal{V}_I (\lambda))\neq 0$. Similarly, let $C_{\zip}$ be the set of $\lambda$ such that $H^0(\GZip^\mu,\mathcal{V}_I (\lambda))\neq 0$ (one can show that they are cones in $X^*(T)$). The inclusion \eqref{inclusion-intro} shows that $C_{\zip} \subset C_{K}$. Denote by $(-)_{\QQ_{>0}}$ the generated $\QQ_{>0}$-cones. Then one can see (\cite[Corollary 1.5.3]{Koskivirta-automforms-GZip}) that $ C_{K,\QQ_{>0}}$ is independent of $K$, and we conjecture (\cite[Conjecture 2.1.6]{Goldring-Koskivirta-global-sections-compositio}) that it coincides with $C_{\zip,\QQ_{>0}}$. Goldring and the second author proved this conjecture in some case in \cite[Theorem D]{Goldring-Koskivirta-global-sections-compositio}. 

We show that the space $H^0(\GZip^\mu,\mathcal{V}(\rho))$ is given by the intersection of the $L_{\varphi}$-invariants of $V$ with a generalized Brylinski--Kostant filtration (where $L_\varphi\subset L$ is a certain $0$-dimensional group, see \eqref{Lphi-equ}). 
For the general statement, see Theorem \ref{thm-main}. For the sake of brevity, we give a simplified statement in this introduction. Assume here that $P$ is defined over $\FF_q$ (in this case, $L_\varphi=L(\FF_q)$). Let $\wp^* \colon  X^*(T)_\RR \to X^*(T)_\RR$ be the map induced by the Lang torsor $\wp \colon T \to T; \ g \mapsto g \varphi(g)^{-1}$, where $\varphi  \colon G\to G$ denotes the $q$-th power Frobenius homomorphism. Let $V=\bigoplus_\nu V_\nu$ be the weight decomposition of $V$. For $\chi\in X^*(T)_\RR$, let $\fil_{\chi}^P V_\nu$ be the Brylinski--Kostant filtration of $V_\nu$ (see \eqref{equ-FilPchi}). 

\begin{thm1}[Corollary \ref{cor-main-Fq}]\label{intro-thm}
Assume that $P$ is defined over $\FF_q$. For any $(V,\rho)\in \Rep(P)$, we have
\begin{equation*}
H^0(\GZip^\mu,\mathcal{V}(\rho))=V^{L(\FF_q)}\cap 
 \bigoplus_{\nu \in X^*(T)} 
 \fil_{{\wp^*}^{-1}(\nu)}^P 
 V_{\nu}.
\end{equation*}
\end{thm1}

In the more simple case of \cite{Koskivirta-automforms-GZip}, the space $V_I (\lambda)_{\leq 0}$ appearing in the equation \eqref{intor-formula-Vlambda} above is a sum of weight spaces of $V$. In the general case, $H^0(\GZip^\mu,\mathcal{V}(\rho))$ cannot be written as an intersection of $V^{L(\FF_q)}$ with a sum of weight spaces of $V$ (see Examples \ref{example-notsum} for a counter-example).
We include examples of concrete computations of the space $H^0(\GZip^\mu,\mathcal{V}(\rho))$ in \S \ref{sec-examples}.

Our second result concerns the category $\VB(\GZip^\mu)$ of vector bundles on $\GZip^\mu$. As explained above, there is a natural functor $\mathcal{V} \colon \Rep(P) \to \VB(\GZip^\mu)$. 
Denote by $\VB_P(\GZip^\mu)$ the full  subcategory which is equal to the essential image of $\mathcal{V}$. 
We give an explicit description of the category $\VB_P(\GZip^\mu)$ of automorphic vector bundles. 
We define the category of $L_{\varphi}$-modules with $\Delta^P$-monodromy (see Definition \ref{deffilmod}). Its objects are $L_{\varphi}$-modules $W$ endowed with a set of monodromy operators indexed by $\Delta^P$ (where $\Delta^P$ denotes the set of simple roots outside the parabolic $P$). There is a natural functor $F_{\mathrm{MN}} \colon \Rep(P)\to L_{\varphi}\mathchar`-\mathrm{MN}_{\Delta^P}$ (see \eqref{equ-functor-modfil}). An $L_{\varphi}$-module with $\Delta^P$-monodromy is called admissible if it lies in the essential image of $F_{\mathrm{MN}}$. The category of admissible $L_{\varphi}$-modules $\Delta^P$-monodromy is denoted by  $L_{\varphi}\mathchar`-\mathrm{MN}_{\Delta^P}^{\mathrm{adm}}$.

\begin{thm2}[Theorem \ref{equivNmod}]
The functor $\mathcal{V} \colon \Rep(P)\to \VB(\GZip^\mu)$ factors through the functor 
$F_{\mathrm{MN}} \colon \Rep(P)\to L_{\varphi}\mathchar`-\mathrm{MN}_{\Delta^P}^{\mathrm{adm}}$ and induces an equivalence of categories
\begin{equation}
 L_{\varphi}\mathchar`-\mathrm{MN}_{\Delta^P}^{\mathrm{adm}} \longrightarrow \VB_P(\GZip^\mu).
\end{equation}
\end{thm2}

In particular, we deduce the following. Let $S_K$ denote again the good reduction special fiber of a Hodge-type Shimura variety. Similarly, there is a natural functor $\Rep(P)\to \VB(S_K)$, where $\VB(S_K)$ denotes the category of vector bundles on $S_K$. Write again $\VB_P(S_K)$ for the essential image of $\Rep(P)$. 
In this context, we have the following:

\begin{cor3}[Corollary \ref{cor-factorize-Shimura}]
The functor $\mathcal{V} \colon \Rep(P)\to \VB_P(S_K)$ factors as 
\begin{equation}
\xymatrix@1@M=5pt{
\Rep(P) \ar[r]^-{F_{\mathrm{MN}}} &  L_{\varphi}\mathchar`-\mathrm{MN}_{\Delta^P}^{\mathrm{adm}} \ar[r]^-{\zeta^*} & \VB_P(S_K).
}
\end{equation}
\end{cor3}
The results of this paper will be used in the follow-up articles \cite{Imai-Koskivirta-partial-Hasse} and \cite{Goldring-Imai-Koskivirta-weights}, where we study partial Hasse invariants for Shimura varieties of Hodge-type. 

\section{Vector bundles on the stack of $G$-zips}

\subsection{Notation}

Throughout the paper, $p$ is a prime number, $q$ is a power of $p$ and $\FF_q$ is the finite field with $q$ elements. We write $k=\overline{\FF}_q$ for an algebraic closure of $\FF_q$. 
Write $\sigma \in \Gal(k/\FF_q)$ for the $q$-th power Frobenius. 
For a $k$-scheme $X$ and $m \in \ZZ$, we write  $X^{(q^m)}$ for the base change of $X$ by $\sigma^m$ and $\varphi \colon X^{(q^m)}\to X^{(q^{m+1})}$ for the relative $q$-th power Frobenius morphism. 
For an algebraic representation $(V,\rho)$ of an algebraic group $H$ over $k$, 
let $(V^{(q)},\rho^{(q)})$ 
denote the representation 
$\rho \circ \varphi \colon H^{(q^{-1})} \to H \to \GL (V)$. 

The notation $G$ will denote a connected reductive group over $\FF_q$. 
We will always write $(B,T)$ for a Borel pair defined over $\FF_q$, i.e. $T \subset B \subset G_k$ are a maximal torus and a Borel subgroup defined over $\FF_q$. 
Let $B^+$ be the Borel subgroup of $G_k$ opposite to $B$ with respect to $T$ (i.e. the unique Borel subgroup of $G$ such that $B^+\cap B=T$). We will use the following notations:
\begin{itemize}
\item As usual, $X^*(T)$ (resp. $X_*(T)$) denotes the group of characters (resp. cocharacters) of $T$. The group $\Gal(k/\FF_q)$ acts naturally on these groups. Let $W=W(G_k,T)$ be the Weyl group of $G_k$. Similarly, $\Gal(k/\FF_q)$ acts on $W$. Furthermore, the actions of $\Gal(k/\FF_q)$ and $W$ on $X^*(T)$ and $X_*(T)$ are compatible in a natural sense.
\item $\Phi\subset X^*(T)$: the set of $T$-roots of $G$.
\item $\Phi_+\subset \Phi$: the system of positive roots with respect to $B^+$ (i.e. $\alpha \in \Phi_+$ when the $\alpha$-root group $U_{\alpha}$ is contained in $B^+$). This convention may differ from other authors. We use it to match the conventions of \cite[II.1.8]{jantzen-representations} and previous publications \cite{Goldring-Koskivirta-Strata-Hasse}, \cite{Koskivirta-automforms-GZip}.
\item $\Delta\subset \Phi_+$: the set of simple roots. 
\item For $\alpha \in \Phi$, let $s_\alpha \in W$ be the corresponding reflection. The system $(W,\{s_\alpha\}_{\alpha \in \Delta})$ is a Coxeter system. Write $\ell  \colon W\to \NN$ for the length function. Hence $\ell(s_\alpha)=1$ for all $\alpha\in \Phi$. Let $w_0$ denote the longest element of $W$.
\item For a subset $K\subset \Delta$, let $W_K$ denote the subgroup of $W$ generated by $\{s_\alpha \}_{\alpha \in K }$. Write $w_{0,K}$ for the longest element in $W_K$.
\item Let ${}^KW$ denote the subset of elements $w\in W$ which have minimal length in the coset $W_K w$. Then ${}^K W$ is a set of representatives of $W_K\backslash W$. The longest element in the set ${}^K W$ is $w_{0,K} w_0$.
\item $X_{+}^*(T)$ denotes the set of dominant characters, i.e. characters $\lambda\in X^*(T)$ such that $\langle \lambda,\alpha^\vee \rangle \geq 0$ for all $\alpha \in \Delta$.
\item For a subset $I\subset \Delta$, let $X_{+,I}^*(T)$ denote the set of characters $\lambda\in X^*(T)$ such that $\langle \lambda,\alpha^\vee \rangle \geq 0$ for all $\alpha \in I$. We call them $I$-dominant characters.
\end{itemize}

\begin{definition}\label{def-IP}
Let $P\subset G_k$ be a parabolic subgroup containing $B$ and let $L\subset P$ be the unique Levi subgroup of $P$ containing $T$. Then we define a subset $I_P\subset \Delta$ as the unique subset such that $W(L,T)=W_{I_P}$. For an arbitrary parabolic subgroup $P\subset G_k$ containing $T$, we put $I_P = I_{P'} \subset \Delta$ where $P'$ is the unique conjugate of $P$ containing $B$.
\end{definition}

\begin{itemize}
    \item For a parabolic $P\subset G_k$, we put $\Delta^P = \Delta \setminus I_P$.
\end{itemize}

\subsection{The stack of $G$-zips}

In this section, we recall some facts about the stack of $G$-zips of Pink--Wedhorn--Ziegler.

\subsubsection{Zip datum} \label{subsec-zipdatum}
Let $G$ be a connected reductive group over $\FF_q$. In this paper, a zip datum is a tuple $\Zcal = (G,P,L,Q,M,\varphi)$ consisting of the following objects:
\begin{enumerate}[(i)]
    \item $P\subset G_k$ and $Q\subset G_k$ are parabolic subgroups of $G_k$.
    \item $L\subset P$ and $M\subset Q$ are Levi subgroups such that $L^{(q)}=M$. In particular, the $q$-power Frobenius isogeny induces an isogeny $\varphi \colon L\to M$.  
\end{enumerate}

If $H$ is an algebraic group, denote by $R_{\mathrm{u}}(H)$ the unipotent radical of $H$. For $x\in P$, we can write uniquely $x=\overline{x}u$ with $\overline{x}\in L$ and $u\in R_{\mathrm{u}}(P)$. This defines a projection map $\theta^P_L \colon P\to L;\ x\mapsto \overline{x}$. 
Similarly, we have a projection $\theta^Q_M \colon Q\to M$. The zip group is the subgroup of $P\times Q$ defined by
\begin{equation}\label{zipgroup}
E \colonequals \{ (x,y)\in P\times Q \mid  \varphi(\theta^P_L(x))=\theta^Q_M(y) \}. 
\end{equation}
In other words, $E$ is the subgroup of $P\times Q$ generated by $R_{\mathrm{u}}(P)\times R_{\mathrm{u}}(Q)$ and elements of the form $(a,\varphi(a))$ with $a\in L$. Let $G\times G$ act on $G_k$ by $(a,b)\cdot g \colonequals agb^{-1}$, and let $E$ act on $G$ by restricting this action to $E$. The stack of $G$-zips of type $\Zcal$ can be defined as the quotient stack
\[
\GZip^\Zcal = \left[E\backslash G_k \right].
\]
Although the above definition of $\GZip^\Zcal$ may be the most concise one, there is more useful, equivalent definition in terms of torsors: By \cite[3C and 3D]{PinkWedhornZiegler-F-Zips-additional-structure}, the stack $\GZip^\Zcal$ is the stack over $k$ such that for all $k$-scheme $S$, the groupoid $\GZip(S)$
is the category of tuples $\underline{\mathcal{I}}=(\mathcal{I},\mathcal{I}_P,\mathcal{I}_Q,\iota)$, where $\mathcal{I}$ is a $G_k$-torsor over $S$, $\mathcal{I}_P\subset \mathcal{I}$ and $\mathcal{I}_Q\subset \mathcal{I}$ are a $P$-subtorsor and a $Q$-subtorsor of $\mathcal{I}$ respectively, and $\iota \colon (\mathcal{I}_P/R_{\mathrm{u}}(P))^{(p)}\to \mathcal{I}_Q/R_{\mathrm{u}}(Q)$ is an isomorphism of $M$-torsors.

\subsubsection{Cocharacter datum} \label{subsec-cochar}
A convenient way to give a zip datum is using cocharacters. A \emph{cocharacter datum} is a pair $(G,\mu)$ where $G$ is a reductive connected group over $\FF_q$ and $\mu \colon \GG_{\mathrm{m},k}\to G_k$ is a cocharacter. There is a natural way to attach to $(G,\mu)$ a zip datum $\Zcal_\mu$, defined as follows. First, denote by $P_+(\mu)$ (resp. $P_-(\mu)$) the unique parabolic subgroup of $G_k$ such that $P_+(\mu)(k)$ (resp. $P_-(\mu)(k)$) consists of the elements $g\in G(k)$ satisfying that the map 
\[
\GG_{\mathrm{m},k} \to G_{k}; \  t\mapsto\mu(t)g\mu(t)^{-1} \quad (\textrm{resp. } t\mapsto\mu(t)^{-1}g\mu(t))
\]
extends to a morphism of varieties $\AA_{k}^1\to G_{k}$. This construction yields a pair of parabolics $(P_+(\mu),P_{-}(\mu))$ in $G_k$ such that the intersection $P_+(\mu)\cap P_-(\mu)=L(\mu)$ is the centralizer of $\mu$. It is a common Levi subgroup of $P_+(\mu)$ and $P_-(\mu)$. Set $P = P_-(\mu)$, $Q = (P_+(\mu))^{(q)}$, $L = L(\mu)$ and $M = (L(\mu))^{(q)}$. Then the tuple $\Zcal_\mu \colonequals (G,P,L,Q,M,\varphi)$ is a zip datum, which we call the zip datum attached to the cocharacter datum $(G,\mu)$. We write simply $\GZip^\mu$ for $\GZip^{\Zcal_\mu}$. For simplicity, we will always consider zip data arising in this way from a cocharacter datum.

\subsubsection{Frames} \label{sec-frames}
In this paper, given a zip datum $\Zcal=(G,P,L,Q,M,\varphi)$, a frame for $\Zcal$ is a triple $(B,T,z)$ where $(B,T)$ is a Borel pair of $G_k$ defined over $\FF_q$ satisfying the following conditions
\begin{enumerate}[(i)]
    \item One has the inclusion $B\subset P$.
    \item $z\in W$ is an element satisfying the conditions
    \begin{equation}\label{eqBorel}
{}^z \! B \subset Q \quad \textrm{and} \quad
 B\cap M= {}^z \! B\cap M. 
\end{equation}
\end{enumerate}
\begin{rmk}\label{rmk-frame}
Let $(B,T)$ be a Borel pair defined over $\FF_q$ such that $B\subset P$. Then we can find $z\in W$ such that $(B,T,z)$ is a frame. This follows from the proof of \cite[Proposition 3.7]{Pink-Wedhorn-Ziegler-zip-data}.
\end{rmk}

A frame may not always exist. However, if $(G,\mu)$ is a cocharacter datum and $\Zcal_\mu$ is the associated zip datum (\S\ref{subsec-cochar}), then we can find a $G(k)$-conjugate $\mu'=\ad(g)\circ \mu$ (with $g\in G(k)$) such that $\Zcal_{\mu'}$ admits a frame. This follows easily from Remark \ref{rmk-frame} and the fact that $G$ is quasi-split over $\FF_q$. Hence, it is harmless to assume that a frame exists, and we will only consider a zip datum that admits a frame.

\begin{rmk}\label{rmkmuFp}
If the cocharacter $\mu$ is defined over $\FF_q$, then so are $P$ and $Q$. In particular, we have in this case $L=M$ and $P$, $Q$ are opposite parabolic subgroups with common Levi subgroup $L$.
\end{rmk}

For a zip datum $(G,P,L,Q,M,\varphi)$, we put 
$I = I_P \subset \Delta$. 
Note that $\Delta^P =\Delta \setminus I$.

\begin{lemma}[{\cite[Lemma 2.3.4]{Goldring-Koskivirta-zip-flags}}]\label{lem-framemu}
Let $\mu \colon \GG_{\mathrm{m},k}\to G_k$ be a cocharacter, and let $\Zcal_\mu$ be the attached zip datum. Assume that $(B,T)$ is a Borel pair defined over $\FF_q$ such that $B\subset P$. We put 
$z =\sigma(w_{0,I})w_0$. 
Then $(B,T,z)$ is a frame for $\Zcal_\mu$.
\end{lemma}

\subsubsection{Parametrization of $E$-orbits} \label{subsec-zipstrata}
Recall that the group $E$ from \eqref{zipgroup} acts on $G_k$. 
We review below the parametrization of $E$-orbits following \cite{Pink-Wedhorn-Ziegler-zip-data}.

Assume that $\Zcal$ has a frame $(B,T,z)$. 
For $w\in W$, fix a representative $\dot{w}\in N_G(T)$, such that $(w_1w_2)^\cdot = \dot{w}_1\dot{w}_2$ whenever $\ell(w_1 w_2)=\ell(w_1)+\ell(w_2)$ (this is possible by choosing a Chevalley system, \cite[ XXIII, \S6]{SGA3}). For $w\in W$, define $G_w$ as the $E$-orbit of $\dot{w}\dot{z}^{-1}$. 
We note that $G_w$ is independent of 
the chooices of $\dot{w}$ and a frame 
by \cite[Proposition 5.8]{Pink-Wedhorn-Ziegler-zip-data}. 
If no confusion occurs, we write $w$ instead of $\dot{w}$. Define a twisted order on ${}^I W$ as follows. For $w,w'\in {}^I W$, write $w'\preccurlyeq w$ if there exists $w_1\in W_L$ such that $w'\leq w_1 w \sigma(w_1)^{-1}$. This defines a partial order on ${}^I W$ (\cite[Corollary 6.3]{Pink-Wedhorn-Ziegler-zip-data}).

\begin{theorem}[{\cite[Theorem 6.2, Theorem 7.5]{Pink-Wedhorn-Ziegler-zip-data}}] \label{thm-E-orb-param}
The map $w\mapsto G_w$ restricts to a bijection
\begin{equation} \label{orbparam}
{}^I W \to \{E \textrm{-orbits in }G_k \}.
\end{equation}
For $w\in {}^I W$, one has $\dim(G_w)= \ell(w)+\dim(P)$. Furthermore, for $w\in {}^I W$, the Zariski closure of $G_w$ is 
\begin{equation}\label{equ-closure-rel}
\overline{G}_w=\bigsqcup_{w'\in {}^IW,\  w'\preccurlyeq w} G_{w'}.
\end{equation}
\end{theorem}

Each $E$-orbit is locally closed in $G_k$. Since $E$ is smooth over $k$, all $E$-orbits are also smooth over $k$. However, the Zariski closure $\overline{G}_w$ of $G_w$ may have highly complicated singularities, see \cite{Koskivirta-Normalization} for a description of the normalization of $\overline{G}_w$. The closure of an $E$-orbit is a union of $E$-orbits, hence we obtain a stratification of $G$. 

In particular, there is a unique open $E$-orbit
$U_\Zcal\subset G_k$ corresponding to the longest element $w_{0,I}w_0\in {}^I W$ via \eqref{orbparam}. For an $E$-orbit $G_w$ (with $w\in {}^I W$), we write $\Xcal_w \colonequals [E\backslash G_w]$ for the corresponding locally closed substack of $\GZip^\Zcal=[E\backslash G_k]$. 

If $\Zcal$ arises from a cocharacter datum (\S\ref{subsec-cochar}), we write $U_\mu$ for $U_{\Zcal_\mu}$. Using the terminology pertaining to the theory of Shimura varieties, we call $U_\mu$ the \emph{$\mu$-ordinary stratum} of $\GZip^\mu$.  The corresponding substack $\Ucal_\mu \colonequals [E\backslash U_\mu]$ is called the \emph{$\mu$-ordinary locus}. It corresponds to the $\mu$-ordinary locus in the good reduction of Shimura varieties, studied for example in \cite{Wortmann-mu-ordinary}, \cite{Moonen-Serre-Tate}. For more details about Shimura varieties, we refer to \S\ref{subsec-Shimura} below. 

\subsection{Reminders about representation theory}\label{subsec-remind}

If $H$ is an algebraic group over a field $K$, denote by $\Rep(H)$ the category of algebraic representations of $H$ on finite-dimensional $K$-vector spaces. 
We will denote such a representation by $(V,\rho)$, or sometimes simply $\rho$ or $V$. 

Let $H$ be a split connected reductive $K$-group and choose a Borel pair $(B_H,T)$ defined over $K$. Irreducible representations of $H$ are in 1-to-1 correspondence with dominant characters $X^*_+(T)$. This bijection is given by the highest weight of a representation. 
For $\lambda\in X_{+}^*(T)$, let $\Lcal_\lambda$ be the line bundle attached to $\lambda$ on the flag variety $H/B_H$ by the usual associated sheaf construction (\cite[\S5.8]{jantzen-representations}). Define an $H$-representation $V_H(\lambda)$ by 
\begin{equation}\label{Vlambdadef}
    V_H(\lambda) \colonequals H^0(H/B_H,\Lcal_\lambda).
\end{equation}
In other words, $V_H(\lambda)$ is the induced representation $\Ind_{B_H}^{H} \lambda$. Then $V_H(\lambda)$ is a representation of highest weight $\lambda$. 
We view elements of $V_H(\lambda)$ as functions $f \colon H\to \AA^1$ satisfying the relation 
\begin{equation}\label{Vlambda-function}
f(hb)=\lambda(b^{-1})f(h), \quad \forall h\in H, \ \forall b\in B_H.    
\end{equation}
For dominant characters $\lambda,\lambda'$, there is a natural surjective map
\begin{equation}\label{Vlambdamap}
V_H(\lambda)\otimes V_H(\lambda')\to V_H(\lambda+\lambda').
\end{equation}
In the description given by \eqref{Vlambda-function}, this map is simply given by mapping $f\otimes f'$ (where $f\in V_H(\lambda)$, $f'\in V_H(\lambda')$) to the function $ff'\in V_H(\lambda+\lambda')$.

Denote by $W_H \colonequals W(H,T)$ the Weyl group and $w_{0,H}\in W_H$ the longest element. Then $V_H(\lambda)$ has a unique $B_H$-stable line, which is a weight space for the weight $w_{0,H}\lambda$.

\subsection{Vector bundles on the stack of $G$-zips} \label{sec-vector-bundles-gzipz}

\subsubsection{General theory} \label{sec-genth-vb}
For an algebraic stack $\Xcal$, write $\VB(\Xcal)$ for the category of vector bundles on $\Xcal$. Let $X$ be a $k$-scheme and $H$ an affine $k$-group scheme acting on $X$. If $\rho \colon H\to \GL(V)$ is a finite dimensional algebraic representation of $H$, it gives rise to a vector bundle $\mathcal{V}_{H,X}(\rho)$ on the stack $[H\backslash X]$. This vector bundle can be defined geometrically as $[H\backslash (X\times_k V)]$ where $H$ acts diagonally on $X\times_k V$. We obtain a functor 
\[
\mathcal{V}_{H,X} \colon \Rep(H)\to \VB([H\backslash X]).
\]
In particular, similarly to the usual associated sheaf constrution \cite[I.5.8.(1)]{jantzen-representations}, the space of global sections $H^0([H\backslash X],\mathcal{V}_{H,X}(\rho))$ is identified with: 
\begin{equation}\label{globquot}
H^0([H\backslash X],\mathcal{V}_{H,X}(\rho))=\left\{f \colon X\to V \mid f(h \cdot x)=\rho(h) f(x) ,  \quad \forall h\in H, \ \forall x\in X\right\}.
\end{equation}

\subsubsection{Automorphic Vector bundles on $\GZip^\Zcal$}\label{sec-VBGzip}
Fix a zip datum $\Zcal=(G,P,L,Q,M,\varphi)$ and a frame $(B,T,z)$ as usual. By the previous paragraph, we obtain a functor $\mathcal{V}_{E,G} \colon \Rep(E)\to \VB(\GZip^\Zcal)$, that we simply denote by $\mathcal{V}$. For $(V,\rho)\in \Rep(E)$, the space of global sections of $\mathcal{V}(\rho)$ is 
\[
H^0(\GZip^\Zcal,\mathcal{V}(\rho))=\left\{f \colon G_k \to V \mid  f(\epsilon \cdot g)=\rho(\epsilon) f(g) , \quad \forall\epsilon\in E, \ \forall g\in G_k \right\}.
\]
One has the following easy lemma, which follows from the fact that $G_k$ admits an open dense $E$-orbit (see discussion below Theorem \ref{thm-E-orb-param}).

\begin{lemma}[{\cite[Lemma 1.2.1]{Koskivirta-automforms-GZip}}]
Let $(V,\rho)$ be an $E$-representation. Then we have $\dim H^0(\GZip^\Zcal,\mathcal{V}(\rho)) \leq \dim (V)$.
\end{lemma}

The first projection $p_1 \colon E\to P$ induces a functor $p_1^* \colon \Rep(P)\to \Rep(E)$. If $(V,\rho)\in \Rep(P)$, we write again $\mathcal{V}(\rho)$ for $\mathcal{V}(p_1^*(\rho))$. 
Let $\VB_P(\GZip^\Zcal)$ be the essentail image of 
$\mathcal{V} \colon \Rep(P)\to \VB(\GZip^\Zcal)$. 
We call $\VB_P(\GZip^\Zcal)$ 
the category of automorphic vector bundles (\cf \cite[III. Remark 2.3]{Milne-ann-arbor}).  
The goal of this paper is to study the vector bundles $\mathcal{V}(\rho)$ on $\GZip^\Zcal$ and determine their properties for $\rho \in \Rep (P)$. In particular, we seek to understand the properties of $\mathcal{V}(\rho)$ in terms of the representation $(V,\rho)$ defining it.

\subsubsection{$L$-representations}\label{sec-Lreps}

Let $\theta^P_L \colon P\to L$ denote again the natural projection modulo the unipotent radical $R_{\mathrm{u}}(P)$, as in \S\ref{subsec-zipdatum}. It induces by composition a functor 
\begin{equation}
(\theta^P_L)^* \colon \Rep(L)\to \Rep(P).  
\end{equation}
It is easy to see that $(\theta^P_L)^*$ is a fully faithful functor, and its image is the full subcategory of $\Rep(P)$ of $P$-representations which are trivial on $R_{\mathrm{u}}(P)$. Hence, we view $\Rep(L)$ as a full subcategory of $\Rep(P)$. If $(V,\rho)\in \Rep(L)$, we write again $\mathcal{V}(\rho) \colonequals \mathcal{V}((\theta^P_L)^*(\rho))$. For $\lambda\in X^*_{+,I}(T)$, write $B_L \colonequals B\cap L$ and 
define an $L$-representation as 
\begin{equation}\label{equ-Vlambda}
V_I (\lambda)=\Ind_{B_L}^L(\lambda). 
\end{equation}
This is the representation defined in \eqref{Vlambdadef} for $H=L$ and $B_H=B_L$. Denote by $\mathcal{V}_I (\lambda)$ the vector bundle on $\GZip^\Zcal$ attached to $V_I (\lambda)$. We call $\Vcal_I(\lambda)$ the \emph{automorphic vector bundle associated to the weight $\lambda$} on $\GZip^\Zcal$. This terminology stems from Shimura varieties (see \S\ref{subsec-Shimura} below for further details). 
Note that if $\lambda\in X^*(T)$ is not $L$-dominant, then $V_I (\lambda)=0$ and hence $\mathcal{V}_I (\lambda)=0$. 
In \cite{Koskivirta-automforms-GZip}, the second author studied the vector bundles $\mathcal{V}_I (\lambda)$ on $\GZip^\Zcal$. In particular, he investigated the question of determining the set $C_{\zip}$ of characters $\lambda\in X_{+,I}^*(T)$ such that the space $H^0(\GZip^\Zcal,\mathcal{V}_I (\lambda))$ is non-zero. In a work in progress \cite{Goldring-Imai-Koskivirta-weights} with Goldring, we completely determine $C_{\zip}$ under the condition that $P$ is defined over $\FF_q$ and the Frobenius $\sigma$ acts on $I$ by $-w_{0,I}$.

\subsection{Shimura varieties}\label{subsec-Shimura}

In this subsection, we explain the link between the stack of $G$-zips and Shimura varieties. 
Let $\gx$ be a Shimura datum \cite[2.1.1]{Deligne-Shimura-varieties}. In particular, $\mathbf{G}$ is a connected reductive group over $\QQ$. Furthermore, $\mathbf{X}$ provides a well-defined $\mathbf{G}(\overline{\QQ})$-conjugacy class $\{\mu\}$ of cocharacters of $\mathbf{G}_{\overline{\QQ}}$. Write $\mathbf{E}=\egx$ for the reflex field of $\gx$ (i.e. the field of definition of $\{\mu\}$) and $\Ocal_\mathbf{E}$ for its ring of integers. 
Given an open compact subgroup $K \subset \gofaf$, write $\shgx_{K}$ for the canonical model at level $K$ over $\mathbf{E}$ (\cf  \cite[2.2]{Deligne-Shimura-varieties}). For $K$ small enough in $\mathbf{G}(\AA_f)$, $\shgx_K$ is a smooth, quasi-projective scheme over $\mathbf{E}$. 
For a small enough $K$, every inclusion $K' \subset K$ induces a finite \'{e}tale projection $\pi_{K'/K} \colon \shgx_{K'} \to \shgx_{K}$.

Let $g\geq 1$ and let $(V,\psi)$ be a $2g$-dimensional, non-degenerate symplectic space over $\QQ$. Write $\GSp(2g)=\GSp(V, \psi)$ for the group of symplectic similitudes of $(V,\psi)$. Write $\mathbf{X}_g$ for the double Siegel half-space \cite[1.3.1]{Deligne-Shimura-varieties}. The pair $\shdagsp$ is called the Siegel Shimura datum and has reflex field $\QQ$. Recall that $\gx$ is of Hodge type if there exists an embedding of Shimura data $\iota \colon \gx \hookrightarrow \shdagsp$ for some $g \geq 1$. Henceforth, assume $\gx$ is of Hodge-type.

Fix a prime number $p$, and assume that the level $K$ is of the form $K=K_pK^p$ where $K_p\subset \mathbf{G}(\QQ_p)$ is a hyperspecial subgroup and $K^p\subset \mathbf{G}(\AA_f^p)$ is an open compact subgroup. 
Recall that a hyperspecial subgroup of $\mathbf{G}(\QQ_p)$ exists if and only if $\mathbf{G}_{\QQ_p}$ is unramified, and is of the form $K_p=\mathscr{G}(\ZZ_p)$ where $\mathscr{G}$ is a reductive group over $\ZZ_p$ such that $\mathscr{G}\otimes_{\ZZ_p}\QQ_p\simeq \mathbf{G}_{\QQ_p}$ and $\mathscr{G}\otimes_{\ZZ_p}\FF_p$ is connected.

We assume that $p >2$. 
For any place $v$ above $p$ in $\mathbf{E}$, Kisin (\cite{Kisin-Hodge-Type-Shimura}) and Vasiu (\cite{Vasiu-Preabelian-integral-canonical-models})  constructed a family of smooth $\Ocal_{\mathbf{E}_v}$-schemes $\Sscr=(\Sscr_K)_{K^p}$, where $K=K_pK^p$ and $K^p$ is a small enough compact open subgroup of $\mathbf{G}(\AA_f^p)$. For $K'^p\subset K^p$, one has again a finite \'{e}tale projection $\pi_{K'/K} \colon \Sscr_{K_p K'^p}\to \Sscr_{K_pK^p}$, where $K=K_pK^p$ and $K'=K_pK'^p$, 
and the tower $\Sscr=(\Sscr_K)_{K^p}$ is an $\Ocal_{\mathbf{E}_v}$-model of the tower $(\shgx_{K})_{K^p}$. 
We write $S_K$ for the geometric special fiber of $\Sscr_K$.

We take a representative $\mu\in \{\mu\}$ defined over $\mathbf{E}_v$ by \cite[(1.1.3) Lemma (a)]{Kottwitz-Shimura-twisted-orbital}. 
We can also assume that $\mu$ extends to $\mu \colon \GG_{\mathrm{m},\Ocal_{\mathbf{E}_v}}\to \mathscr{G}_{\Ocal_{\mathbf{E}_v}}$ (\cite[Corollary 3.3.11]{Kim-Rapoport-Zink-uniformization}). 
Denote by $\mathbf{L} \subset \mathbf{G}_{\mathbf{E}_v}$ the centralizer of the cocharacter $\mu$. We take a 
parabolic sugroups $\mathbf{P}$ of $\mathbf{G}_{\mathbf{E}_v}$,  which has $\mathbf{L}$ as a Levi subgroup. 
Since $\mathbf{G}_{\QQ_p}$ is unramified, it is quasi-split, hence we can choose a Borel subgroup $\mathbf{B}\subset \mathbf{G}_{\QQ_p}$ and a maximal torus $\mathbf{T}\subset \mathbf{B}$. 
There is $g \in \mathbf{G}(\mathbf{E}_v)$ such that 
$\mathbf{B}_{\mathbf{E}_v} \subset g \mathbf{P} g^{-1}$. 
Write $g=b g_0$ with $b \in B(\mathbf{E}_v)$ and $g_0 \in \mathscr{G} (\Ocal_{\mathbf{E}_v})$ by the Iwasawa decomposition. Then replacing $\mu$ by its conjugate by $g_0$, we may assume that 
$\mathbf{B}_{\mathbf{E}_v} \subset \mathbf{P}$. 

By properness of the scheme of parabolic subgroups of $\mathscr{G}$ (\cite[Expos\'{e} XXVI, Corollaire 3.5]{SGA3}), the subgroups $\mathbf{B}$ and $\mathbf{P}$ extend uniquely to subgroups $\mathscr{B} \subset \mathscr{G}$ over $\ZZ_p$ and $\mathscr{P} \subset \mathscr{G}_{\Ocal_{\mathbf{E}_v}}$ over $\Ocal_{\mathbf{E}_v}$ respectively. 
Let $\mathscr{L}\subset \mathscr{P}$ be the centralizer of 
$\mu \colon \GG_{\mathrm{m},\Ocal_{\mathbf{E}_v}}\to \mathscr{G}_{\Ocal_{\mathbf{E}_v}}$. 
We take a Borel subgroup $\mathbf{B}^{\mathrm{op}}$ of $\mathbf{G}_{\QQ_p}$ such that $\mathbf{T}=\mathbf{B} \cap \mathbf{B}^{\mathrm{op}}$. 
The subgroup $\mathbf{B}^{\mathrm{op}}$ extends uniquely to a subgroup $\mathscr{B}^{\mathrm{op}} \subset \mathscr{G}$ over $\ZZ_p$. 
We put 
$\mathscr{T} =\mathscr{B} \cap \mathscr{B}^{\mathrm{op}}$. 
Set $G = \mathscr{G} \otimes_{\ZZ_p} \FF_p$ and denote by $B, T, P, L$ the geometric  special fiber of $\mathscr{B}, \mathscr{T}, \mathscr{P}, \mathscr{L}$ respectively. By slight abuse of notation, we denote again by $\mu$ its mod $p$ reduction $\mu \colon \GG_{\mathrm{m},k}\to G_k$. Then $(G,\mu)$ is a cocharacter datum, and it yields a zip datum $(G,P,L,Q,M,\varphi)$ as in \S\ref{subsec-cochar} (since $G$ is defined over $\FF_p$, in the context of Shimura varieties, we always take $q=p$, hence $\varphi$ is the $p$-th power Frobenius).

By a result of Zhang (\cite[4.1]{Zhang-EO-Hodge}), there exists a natural smooth morphism
\begin{equation}\label{zeta-Shimura}
\zeta \colon S_K\to \GZip^\mu. 
\end{equation}
This map is also surjective by \cite[Corollary 3.5.3(1)]{Shen-Yu-Zhang-EKOR}. 
The map $\zeta$ amounts to the existence of a universal $G$-zip $\underline{\mathcal{I}}=(\mathcal{I},\mathcal{I}_P,\mathcal{I}_Q,\iota)$ over $S_K$, using the description of $\GZip^\mu$ provided at the end of \S\ref{subsec-zipdatum}. In the construction of Zhang, the $G_k$-torsor $\mathcal{I}$ and the $P$-torsor $\mathcal{I}_P$ over $S_K$ are actually the reduction of a $\mathscr{G}$-torsor and a $\mathscr{P}$-torsor over $\Sscr_{K}$, that we denote by $\mathscr{I}$ and $\mathscr{I}_\mathscr{P}$ respectively.

\begin{example}
We explain the example of the Siegel-type Shimura variety. In this case, one has $\mathbf{G}=\GSp(V,\psi)$ for a symplectic space $(V,\psi)$ of dimension $2g$ ($g\geq 1$) over $\QQ$. The $\ZZ_p$-model $\mathscr{G}=\GSp(\Lambda,\psi)$ is given by a self-dual $\ZZ_p$-lattice $\Lambda\subset V_{\QQ_p}$, i.e. a lattice satisfying $\Lambda^\vee=\Lambda$, where $\Lambda^\vee \colonequals \{x\in V_{\QQ_p} \mid \forall y\in \Lambda, \ \psi (x, y) \in \ZZ_p\}$. The cocharacter $\mu \colon \GG_{\mathrm{m},\ZZ_p}\to \mathbf{G}_{\ZZ_p}$ induces a decomposition $\Lambda=\Lambda_0\oplus \Lambda_1$, where $\Lambda_0$, $\Lambda_1$ are free $\ZZ_p$-modules of rank $g$. Here $z\in \GG_{\mathrm{m}}$ acts via $\mu$ on $\Lambda_i$ by the character $z\mapsto z^i$ for $i\in \{0,1\}$. Define two filtrations
\begin{align*}
    \Fil_0(\Lambda)  \colon & \ 0\subset \Lambda_0 \subset \Lambda \quad \textrm{and} \\
    \Fil_1(\Lambda)  \colon & \ 0\subset \Lambda_1 \subset \Lambda.
\end{align*}
Then $\mathscr{P}$ can be defined as the parabolic subgroup of $\mathscr{G}$ stabilizing $\Fil_0(\Lambda)$. The scheme $\Sscr_{K}$ (with $K=K_pK^p$ and $K_p=\mathscr{G}(\ZZ_p)$ as above) is a moduli space classifying triples $(A,\xi, \eta K^p)$ where $A$ is an abelian variety of rank $g$ endowed with a principal polarization $\xi$, and a $K^p$-level structure $\eta K^p$. Here $\eta$ is a symplectic isomorphism $H^1(A,\AA^p)\simeq V\otimes \AA^p$ and $\eta K^p$ is its $K^p$-coset in the set of such isomorphisms.

Let $\mathscr{A}\to \Sscr_{K}$ denote the universal abelian scheme. Then $\mathscr{H} \colonequals H^1_{\dR}(\mathscr{A}/\Sscr_K)$ is a rank $2g$ vector bundle on $\Sscr_K$, and the principal polarization $\xi$ induces on $\mathscr{H}$ a perfect, symplectic pairing, that we denote by $\psi_\xi$. The vector bundle $\mathscr{H}$ also carries a natural Hodge filtration (that we denote by $\Fil_{\hdg}$):
\begin{equation}
    0\subset \Omega_{\mathscr{A}/\Sscr_K}\subset \mathscr{H}
\end{equation}
where $\Omega_{\mathscr{A}/\Sscr_K}$ is the push-forward of the sheaf of relative K\"{a}hler differentials $\Omega^1_{\mathscr{A} / \Sscr_K}$ by the structural morphism $f \colon \mathscr{A} \to \Sscr_K$. It is a rank $g$-subbundle of $\mathscr{H}$. We obtain a $\mathscr{G}$-torsor $\mathscr{I}$ and a $\mathscr{P}$-torsor $\mathscr{I}_\mathscr{P}$ over $\Sscr_K$ as follows: For an $\Sscr_K$-scheme $S$, we 
define $\mathscr{I}(S)$ by 
\begin{align*}
\underline{\Isom}_{\Ocal_S}\left((\Lambda\otimes \Ocal_S,\psi) , (\mathscr{H}\otimes_{\Ocal_{\Sscr_K}}\! \Ocal_S, \psi_\xi)\right), 
\end{align*}
and $\mathscr{I}_\mathscr{P}(S)$ by 
\begin{align*}
\underline{\Isom}_{\Ocal_S}\left((\Lambda\otimes \Ocal_S, \psi, \Fil_0(\Lambda)\otimes \Ocal_S) ,  (\mathscr{H}\otimes_{\Ocal_{\Sscr_K}}\! \Ocal_S,  \psi_\xi,  \Fil_{\hdg} \otimes_{\Ocal_{\Sscr_K}}\! \Ocal_S ) \right).
\end{align*}
This defines two fppf sheaves on $\Sscr_K$. Furthermore $\mathscr{G}$ acts naturally on $\mathscr{I}$ via its action on $\Lambda$. Furthermore, since the parabolic group $\mathscr{P}\subset \mathscr{G}$ stabilizes $\Fil_0(\Lambda)$, the group $\mathscr{P}$ acts naturally on $\mathscr{I}_{\mathscr{P}}$. This defines respectively a $\mathscr{G}$-torsor and a $\mathscr{P}$-torsor on $\Sscr_K$.

Over $S_K=\Sscr_K\otimes \FF_p$, the $G$-zip $\underline{\mathcal{I}}=(\mathcal{I},\mathcal{I}_P,\mathcal{I}_Q,\iota)$ is defined as follows. First define $\mathcal{I}$ and $\mathcal{I}_P$ to be the base change to $S_K$ of $\mathscr{I}$ and $\mathscr{I}_\mathscr{P}$. To define the $Q$-torsor $\mathcal{I}_Q$, recall that $H \colonequals H^1_{\dR}(A/S_K)$ admits a conjugate filtration $\Fil_{\conj}\subset H$: 
Let $f\colon A\to S_K$ denote the universal abelian scheme (with $A \colonequals \mathscr{A}\otimes_{\Sscr_K} S_K$), there is a conjugate spectral sequence $E_2^{ab}=R^af_{*}(\mathcal{H}^b(\Omega^{\bullet}_{A/S_K}))\Rightarrow H^{a+b}_{\dR}(A/S_K)$. For abelian varieties, this spectral sequence degenerates and gives the filtration $\Fil_{\conj}$ on $H^{1}_{\dR}(A/S_K)$. Note that the conjugate filtration only exists on the special fiber of $\Sscr_K$, contrary to the Hodge filtration. 
For an $S_K$-scheme $S$, we put 
\begin{equation}
\mathcal{I}_Q(S) = \underline{\Isom}_{\Ocal_S}\left((\Lambda\otimes \Ocal_S, \psi, \Fil_1(\Lambda) \otimes \Ocal_S)  ,  (H\otimes_{\Ocal_{S_K}}\! \Ocal_S, \psi_\xi, \Fil_{\conj} \otimes_{\Ocal_{S_K}}\! \Ocal_S ) \right).
\end{equation}
Since $Q$ stabilizes the filtration $\Fil_1(\Lambda)\otimes \FF_p$, it acts naturally on $\mathcal{I}_Q$, and again we obtain a $Q$-torsor on $S_K$. Finally, the isomorphism $\iota \colon (\mathcal{I}_P/R_{\mathrm{u}}(P))^{(p)}\to \mathcal{I}_Q/R_{\mathrm{u}}(Q)$ is naturally induced by the Frobenius and Verschiebung homomorphisms (or more generally, the Cartier isomorphism, see \cite[7.3]{Moonen-Wedhorn-Discrete-Invariants}).
\end{example}

For each $\mathbf{L}$-dominant character $\lambda\in X^*(\mathbf{T})$, we have the unique irreducible representation $\mathbf{V}_{I}(\lambda)$ of $\mathbf{P}$ over $\overline{\QQ}_p$ of highest weight $\lambda$. Since we are in characteristic zero, $\mathbf{V}_{I}(\lambda)$ coincides with  $H^0(\mathbf{P}/\mathbf{B},\Lcal_\lambda)$, as defined in \eqref{Vlambdadef} in \S \ref{subsec-remind}. It admits a natural model over $\overline{\ZZ}_p$, namely
\begin{equation}
\mathbf{V}_{I}(\lambda)_{\overline{\ZZ}_p} \colonequals H^0(\mathscr{P}/\mathscr{B},\Lcal_\lambda), 
\end{equation} 
where $\Lcal_\lambda$ is the line bundle attached to 
$\lambda$ viewed as a character of $\mathscr{T}$. 
Its reduction modulo $p$ is the $P$-representation $V_I (\lambda)=H^0(P/B,\Lcal_\lambda)$ over $k=\overline{\FF}_p$. Since $\Sscr_K$ is endowed naturally with a $\mathscr{P}$-torsor $\mathscr{I}_\mathscr{P}$, we obtain a vector bundle $\mathscr{V}_{I}(\lambda)$ on $\Sscr_K$ by applying the $\mathscr{P}$-representation $\mathbf{V}_{I}(\lambda)_{\overline{\ZZ}_p}$ to $\mathscr{I}_\mathscr{P}$. 
The vector bundle $\mathscr{V}_{I}(\lambda)$ for $\lambda\in X^*(\mathbf{T})_{+,I}$ is called \emph{the automorphic vector bundles associated to the weight $\lambda$}. For an $\Ocal_{\mathbf{E}_v}$-algebra $R$, the space $H^0(\Sscr_K\otimes_{\Ocal_{\mathbf{E}_v}} R, \mathscr{V}_{I}(\lambda))$ may be called the space of automorphic forms of level $K$ and weight $\lambda$ with coefficients in $R$. More generally, by the same formalism, we have a commutative diagram of functors
$$\xymatrix{
\Rep_{\overline{\ZZ}_p}(\mathscr{P}) \ar[r]^{\Vscr} \ar[d] & \VB(\Sscr_{K}) \ar[d] \\
\Rep_{\overline{\FF}_p}(P) \ar[r]^{\Vcal} & \VB(S_K)
}$$
where the vertical arrows are reduction modulo $p$ and the horizontal arrows are obtained by applying the $\mathscr{P}$-torsor $\mathscr{I}_{\mathscr{P}}$ and the $P$-torsor $\mathcal{I}_P$ respectively. The vector bundles obtained in this way on $\Sscr_K$ and $S_K$ are called \emph{automorphic vector bundles} following \cite[III. Remark 2.3]{Milne-ann-arbor}.

Furthermore, the map $\zeta  \colon  S_K \to \GZip^\mu$ induces a factorization of the lower horizontal arrow of the above diagram as 
\begin{equation}\label{equ-Shimura-functor}
\xymatrix@1{
\Rep_{\overline{\FF}_p}(P) \ar[r]^-{\mathcal{V}} & \VB(\GZip^\mu) \ar[r]^-{\zeta^*} & \VB(S_K).
}    
\end{equation}
Note also that for any $P$-representation $(V,\rho)$, the map $\zeta  \colon  S_K \to \GZip^\mu$ induces by pull-back a natural injective morphism
\begin{equation}
    H^0(\GZip^\mu, \mathcal{V}(\rho))\to H^0(S_K,\mathcal{V}(\rho)).
\end{equation}
In \S \ref{sec-H0}, we determine the space $H^0(\GZip^\mu, \mathcal{V}(\rho))$ in all generality (i.e. even for cocharacter data $(G,\mu)$ that are not attached to Shimura varieties). For general pairs $(G,\mu)$ with $\mu$ minuscule (but not necessarily attached to Shimura varieties), one has the following remark:

\begin{rmk}\label{rem-shtukas}
Let $F$ be a local field with ring of integers $\Ocal$ and residue field $\FF_q$. Let $G$ be an unramified reductive group over $\Ocal$. Let $(B,T)$ be a Borel pair of $G$, and let $\mu$ be a dominant cocharacter of $G$. Then Xiao--Zhu define the moduli of local shtukas $\Sht^{\loc}_\mu$ classifying modifications bounded by $\mu$ of a $G$-torsor and its Frobenius twist (see \cite[Definition 5.2.1]{Xiao-Zhu-geometric-satake}). Similarly, there is a moduli $\Sht^{\loc(m,n)}_\mu$ of restricted local shtuka (\cite[\S 5.3]{Xiao-Zhu-geometric-satake}), with a natural projection $\Sht^{\loc}_\mu\to \Sht^{\loc(m,n)}_\mu$. In the case when $\mu$ is minuscule, Xiao--Zhu show in \cite[Lemma 5.3.6]{Xiao-Zhu-geometric-satake} that there exists a natural perfectly smooth morphism $\Sht^{\loc(2,1)}_\mu\to \GZip^{\mu,\pf}$, where $\pf$ denotes the perfection and the special fiber of $G$ is again denoted by $G$ (see \S\ref{subsec-perfect} for further details).

\end{rmk}

\section{The space of global sections $H^0(\GZip^\mu, \mathcal{V}(\rho))$}\label{sec-H0}

\subsection{Adapted morphisms}\label{sec-adap}

To determine the space $H^0(\GZip^\mu, \mathcal{V}(\rho))$ (for $(V,\rho)$ a $P$-representation), we use a similar method as in \cite[\S3.2]{Koskivirta-automforms-GZip}, where we studied representations of the type $V_I (\lambda)$. We review some of the notions introduced in \loccit

Let $X$ be an irreducible normal $k$-variety and let $U\subset X$ be an open subset such that $S=X\setminus U$ is irreducible of codimension $1$. For $f\in H^0(U,\Ocal_X)$, denote by $Z_U(f)\subset U$ the vanishing locus of $f$ in $U$ and let $\overline{Z_U(f)}$ be its Zariski closure in $X$. We endow all locally closed subsets of schemes with the reduced structure. Let $Y$ be an irreducible $k$-variety and $\psi \colon Y\to X$ be a $k$-morphism. 
\begin{definition}\label{defadap}
We say that $\psi$ is adapted to $f$ (with respect to $U$) if
\begin{enumerate}[(i)]
\item \label{adap1} $\psi(Y)\cap U\neq \emptyset$, and
\item \label{adap2} $\psi(Y)\cap S$ is not contained in $\overline{Z_U(f)}$.
\end{enumerate}
\end{definition}

\begin{lemma}\label{exadap}
 If $\psi(Y)$ intersects $U$ and $\psi(Y)\cap S$ is dense in $S$, then $\psi$ is adapted to any nonzero section $f\in H^0(U,\Ocal_X)$.
\end{lemma}
\begin{proof}
We need to show that the consition \eqref{adap2} is satisfied. We may assume that $Z_U(f)\neq \emptyset$. Then, the closed subset $\overline{Z_U(f)}$ has codimension 1 in $X$ and intersects $U$, hence $\overline{Z_U(f)}\cap S$ has codimension $\geq 1$ in $S$, so it cannot contain $\psi(Y)\cap S$.
\end{proof}

\begin{lemma}[{\cite[Lemma 3.2.2]{Koskivirta-automforms-GZip}}]\label{adappb}
Let $\psi \colon Y\to X$ be a morphism adapted to $f\in H^0(U,\Ocal_X)$. Then $f$ extends to $X$ if and only if $\psi^*(f)\in H^0(\psi^{-1}(U),\Ocal_Y)$ extends to $Y$. In this case, $f$ vanishes along $S$ if and only if $\psi^*(f)$ vanishes along $\psi^{-1}(S)$.
\end{lemma}

We apply the above notions to the following situation. From now on, let $(G,\mu)$ be a cocharacter datum, with attached zip datum $\Zcal=(G,P,L,Q,M,\varphi)$ as in \S\ref{subsec-cochar}. Assume that $(B,T)$ is a Borel pair defined over $\FF_q$ such that $B\subset P$. We take a frame $(B,T,z)$ as in Lemma \ref{lem-framemu}. 
Consider the variety $G_k$ and the open subset $U_{\mu} \subset G_k$ (the $\mu$-ordinary stratum, defined after Theorem \ref{thm-E-orb-param}). 
The complement of $U_{\mu}$ in $G_k$ is not irreducible in general, so in order to apply the previous results, we slightly modify the problem. Recall the parametrization of $E$-orbits in $G_k$ \eqref{orbparam}. Using Theorem \ref{thm-E-orb-param}, we have
\begin{equation}\label{GminusU}
    G_k \setminus U_{\mu} = \bigcup_{\alpha \in \Delta^P} Z_\alpha, \quad Z_\alpha=\overline{E\cdot s_\alpha}
\end{equation}
where $E\cdot s_\alpha$ denotes the $E$-orbit of $s_{\alpha}$ and the bar denotes the Zariski closure. 
Indeed, by \eqref{orbparam}, the $E$-orbits of codimension $1$ in $G_k$ are the $E$-orbits of $wz^{-1}$ where $w\in {}^I W$ is an element of length $\ell(w_{0,I}w_0)-1$. These elements are of the form $w_{0,I}s_{\alpha}w_0$ for $\alpha \in \Delta^P$. Since $z=\sigma(w_{0,I})w_0$, the element $wz^{-1}$ has the form $w_{0,I}s_{\alpha}\sigma(w_{0,I})$. Since $(w_{0,I},\sigma(w_{0,I}))\in E$, this element generates the same $E$-orbit as $s_\alpha$. This proves the decomposition \eqref{GminusU} above. For any $\alpha\in \Delta^P$, define an open subset 
\begin{equation}
X_\alpha  \colonequals  G_k \setminus \bigcup_{\beta\in \Delta^P,\ \beta\neq \alpha}Z_\beta.
\end{equation}
Clearly $U_{\mu}\subset X_\alpha$ and one has $X_{\alpha}\setminus U_{\mu}=E\cdot s_{\alpha}$. In particular, $X_{\alpha}\setminus U_{\mu}$ is irreducible.
We define a morphism which satisfies the conditions of Definition \ref{defadap} for the pair $(X_\alpha,U_{\mu})$. 

We take an isomorphism 
$u_{\alpha} \colon \GG_{\mathrm{a}} \to U_{\alpha}$ 
for $\alpha \in \Phi$ 
so that 
$(u_{\alpha})_{\alpha \in \Phi}$ is a realization 
in the sense of \cite[8.1.4]{Springer-Linear-Algebraic-Groups-book}. In particular, we have 
\begin{equation}\label{eq:phiconj}
 t u_{\alpha}(x)t^{-1}=u_{\alpha}(\alpha(t)x)
\end{equation}
for $x \in \GG_{\mathrm{a}}$ and $t \in T$. 
For $\alpha \in \Phi$, 
there is a unique homomorphism
\begin{equation}\label{Ralpha}
 \phi_\alpha  \colon  \SL_{2,k} \to G_k 
\end{equation}
such that 
\[
 \phi_\alpha 
 \left( \begin{pmatrix}
 1 & x \\ 0 & 1 
 \end{pmatrix}\right) = u_{\alpha}(x), \quad 
 \phi_\alpha 
 \left( \begin{pmatrix}
 1 & 0 \\ x & 1 
 \end{pmatrix}\right) = u_{-\alpha}(x) 
\]
as in \cite[9.2.2]{Springer-Linear-Algebraic-Groups-book}. Also note that $\phi_{\alpha}(\diag(t,t^{-1}))=\alpha^\vee(t)$. 

Let $\alpha\in \Delta^P$. Set $Y=E \times \AA^1$ and 
\begin{equation}\label{phia}
\psi_\alpha \colon Y \to G_k ; \ ((x,y),t)\mapsto x \phi_{\alpha} \left(A(t)\right)y^{-1} \quad \textrm{where} \ A(t)=\left(
\begin{matrix}
t & 1 \\ -1 & 0 \end{matrix}
\right)\in \SL_{2,k}.
\end{equation}
Note that $\phi_{\alpha} (A(0))=s_\alpha$ in $W$. 
The following identity will be crucial for later purposes:
\begin{equation}\label{At}
A(t)=\left(\begin{matrix}
1&0\\-t^{-1}&1
\end{matrix} \right) \left(\begin{matrix}
t&0\\0&t^{-1}
\end{matrix} \right)\left(\begin{matrix}
1&t^{-1}\\0&1
\end{matrix} \right).
\end{equation}
Let $\wp \colon T \to T; \ g \mapsto g \varphi(g)^{-1}$ be the Lang torsor. 
Then $\wp$ induces the isomorphism 
\begin{equation}\label{def-Pstar}
 \wp_* \colon X_*(T)_{\RR} \stackrel{\sim}{\longrightarrow} X_*(T)_{\RR}; \  \delta \mapsto \wp \circ \delta =\delta - q \sigma(\delta).     
\end{equation}
We put $\delta_{\alpha}=\wp_*^{-1}(\alpha^{\vee})$. 
Recall that $\sigma$ denotes the $q$-th power Frobenius action on $\Delta$. 
We put 
\begin{equation}\label{malpha-equ}
 m_{\alpha} =\min \{ m \geq 1 \mid 
 \sigma^{-m}(\alpha) \notin I \}     
\end{equation}
and $t_{\alpha}=t^{-1}\alpha(\varphi(\delta_{\alpha}(t)))^{-1} =t \alpha(\delta_{\alpha}(t))^{-1} \in t^{\QQ}$, where $t$ is an indeterminate. 

\begin{proposition}\label{psiadapted}
The following properties hold: 
\begin{assertionlist}
\item \label{item-imagepsi} The image of $\psi_\alpha$ is contained in $X_\alpha$.
\item \label{item-psit} For any $(x,y)\in E$ and $t\in \AA^1$, one has $\psi_\alpha((x,y),t)\in U_{\mu} \Longleftrightarrow t\neq 0$.
\item \label{item-zero} For all $(x,y)\in E$, we have $\psi_\alpha((x,y),0)\in E\cdot s_\alpha$.

\end{assertionlist}

\end{proposition}

\begin{proof}
It suffices to show \eqref{item-psit} and \eqref{item-zero}. 
If $t=0$, we have $\phi_{\alpha}(A(0))=s_\alpha$ in $W$. Hence $\psi_\alpha((x,y),0)\in E\cdot s_\alpha$. 
Assume that $t \neq 0$. 
We put 
\[
 u_{t,\alpha}=\prod_{i=1}^{m_{\alpha}-1} \phi_{\sigma^{-i}(\alpha)} 
 \left( \begin{pmatrix}
1&-t_{\alpha}^{\frac{1}{q^i}}\\0&1
\end{pmatrix} \right) 
\]
where the products are taken in the increasing order of indices. 
By \eqref{At} and the definitions of 
$\delta_{\alpha}$, $t_{\alpha}$ and 
$u_{t,\alpha}$, we have 
\begin{align}
 \phi_{\alpha} \left( A(t) \right) 
 &= 
 \phi_{\alpha} \left( \begin{pmatrix}
1&0\\-t^{-1}&1
\end{pmatrix} \right) \delta_{\alpha}(t)\varphi(\delta_{\alpha}(t))^{-1} 
 \phi_{\alpha} \left(\begin{pmatrix}
1&t^{-1} \\ 0&1
\end{pmatrix} \right)\\ 
 &= 
 \phi_{\alpha} \left(\begin{pmatrix}
1&0\\-t^{-1}&1
\end{pmatrix} \right) \delta_{\alpha}(t) 
 \phi_{\alpha} \left(\begin{pmatrix}
1&t_{\alpha} \\ 0&1
\end{pmatrix} \right) \varphi(\delta_{\alpha}(t))^{-1}\\ 
 &=  \phi_{\alpha} \left(\begin{pmatrix}
1&0\\-t^{-1}&1
\end{pmatrix} \right) \delta_{\alpha}(t) 
u_{t,\alpha} 
\left( 
\varphi(\delta_{\alpha}(t))
 \phi_{\alpha} \left(\begin{pmatrix}
1& -t_{\alpha} \\ 0&1
\end{pmatrix} \right)
u_{t,\alpha}
\right)^{-1}. \label{eq:phiadec} 
\end{align}
We have 
\begin{equation}\label{eq:inE}
 \left(
  \phi_{\alpha} \left(\begin{pmatrix}
1&0\\-t^{-1}&1
\end{pmatrix} \right) \delta_{\alpha}(t) 
u_{t,\alpha}, 
\varphi(\delta_{\alpha}(t))
 \phi_{\alpha} \left(\begin{pmatrix}
1& -t_{\alpha} \\ 0&1
\end{pmatrix} \right) 
u_{t,\alpha} 
\right) \in E 
\end{equation}
because 
\begin{equation*}\label{RainRu}
  \phi_{\alpha} \left(\begin{pmatrix}
1&0\\-t^{-1}&1
\end{pmatrix} \right) \in R_{\mathrm{u}}(P), 
 \quad 
 \phi_{\sigma^{-(m_{\alpha}-1)}(\alpha)} 
 \left( \begin{pmatrix}
1&-t_{\alpha}^{\frac{1}{q^{m_{\alpha}-1}}}\\0&1
\end{pmatrix} \right) 
 \in R_{\mathrm{u}}(Q) 
\end{equation*}
by $\alpha \notin I$ and $\sigma^{-(m_{\alpha}-1)}(\alpha) \notin \sigma (I)$. 
Hence we have 
$\psi_\alpha((x,y),t)\in U_{\mu}$ if $t\neq 0$. 
\end{proof}

Set $Y_0 \colonequals E\times \GG_{\mathrm{m}}\subset Y$. We obtain a map $\psi_\alpha \colon Y_0\to U_\mu$.

\begin{corollary}\label{corphial}
Let $f \colon U_{\mu}\to \AA^n$ be a regular map. Then $f$ extends to a regular map $G_k \to \AA^n$ if and only if for all $\alpha \in \Delta^P$, the map $f\circ \psi_\alpha  \colon  Y_0\to \AA^n$ extends to a map $Y\to \AA^n$.
\end{corollary}

\begin{proof}
Applying Lemma \ref{exadap} and Lemma \ref{adappb} to the coordinate functions of $f$, 
we can extend $f$ to $\bigcup_{\alpha \in \Delta^P} X_{\alpha}$. 
Since the complement of $\bigcup_{\alpha \in \Delta^P} X_{\alpha}$ in $G$ has codimension $\geq 2$, 
we can extend $f$ to $G$ by normality. 
\end{proof}

\subsection{The space of $\mu$-ordinary sections} \label{subsec-ordloc}

Recall that $\Ucal_{\mu}=[E\backslash U_{\mu}]\subset \GZip^\mu$ denotes the $\mu$-ordinary locus (see \S\ref{subsec-zipstrata}). The open substack  $\Ucal_\mu\subset \GZip^\mu$  is dense, and hence induces an obvious injective map
\begin{equation}
H^0(\GZip^\mu,\mathcal{V}(\rho))\to H^0(\Ucal_\mu,\mathcal{V}(\rho))
\end{equation}
for any $(V,\rho)\in \Rep(P)$. This will give an upper bound approximation of the space $H^0(\GZip^\mu,\mathcal{V}(\rho))$. We claim that $1\in U_{\mu}$. Indeed, by Theorem \ref{thm-E-orb-param}, $U_\mu$ coincides with the $E$-orbit of the element $w_{0,I}w_0 z^{-1}$. Since $z=\sigma(w_{0,I})w_{0}$, we obtain $w_{0,I}w_0 z^{-1}=w_{0,I}\sigma(w_{0,I})$. This element is in the same $E$-orbit as $1$, because $(w_{0,I},\sigma(w_{0,I}))\in E$. This proves the claim.

We denote by $L_{\varphi}\subset E$ the scheme-theoretical stabilizer of the element $1$. Note that 
\begin{equation}\label{Lphi-equ}
L_{\varphi}=E\cap \{(x,x) \mid x\in G_k \}    
\end{equation}
is a $0$-dimensional algebraic group. In general it is non-smooth. Denote by $L_0\subset L$ the largest algebraic subgroup defined over $\FF_q$. In other words,
\begin{equation}
    L_0=\bigcap_{n\geq 0}L^{(q^n)}.
\end{equation}
In view of \eqref{Lphi-equ}, it is clear that the restriction of the first projection $E\to P$ induces a closed immersion $L_{\varphi}\to P$. Hence we will identify $L_\varphi$ with its image and view it as a subgroup of $P$.

\begin{lemma}[{\cite[Lemma 3.2.1]{Koskivirta-Wedhorn-Hasse}}]\label{lemLphi} \ 
\begin{assertionlist}
\item \label{lemLphi-item1} One has $L_{\varphi}\subset L$.
\item \label{lemLphi-item2}  The group $L_{\varphi}$ can be written as a semidirect product
\begin{equation}
L_{\varphi}=L_{\varphi}^\circ\rtimes L_0(\FF_q)    
\end{equation}
where $L_{\varphi}^\circ$ is the identity component of $L_{\varphi}$. Furthermore, $L_{\varphi}^\circ$ is a finite unipotent algebraic group.
\item \label{lemLphi-item3}  Assume that $P$ is defined over $\FF_q$. Then $L_0=L$ and $L_{\varphi}=L(\FF_q)$, viewed as a constant algebraic group.
\end{assertionlist}
\end{lemma}

\begin{proposition}\label{prop-Umu}
The stack $\Ucal_{\mu}$ is isomorphic to $B(L_{\varphi})=[1/L_{\varphi}]$, the classifying stack of $L_{\varphi}$.
\end{proposition}

\begin{proof}
The action map $E\to U_\mu$, $e\mapsto e\cdot 1$ induces an isomorphism $E/L_{\varphi}\simeq U_\mu$. Hence $\Ucal_\mu =[E\backslash U_\mu]\simeq [E\backslash (E/L_{\varphi})]\simeq [1/L_{\varphi}]$.
\end{proof}

\begin{corollary}\label{cor-global-muord}
The category of vector bundles on $\Ucal_\mu$ is equivalent to the category $\Rep(L_{\varphi})$ of representations of $L_{\varphi}$. Furthermore, for all $(V,\rho)\in \Rep(L_{\varphi})$, the space of global sections of the attached vector bundle $\mathcal{V}(\rho)$ on $\Ucal_\mu$ identifies with the space of $L_\varphi$-invariants of $V$: 
\begin{equation}\label{muordsections}
    H^0(\Ucal_\mu,\mathcal{V}(\rho))=V^{L_{\varphi}}.
\end{equation}
Furthermore, this identification is functorial in $(V,\rho)$.
\end{corollary}
The identity \eqref{muordsections} can be seen as an isomorphism between two functors $\Rep(L_\varphi)\to \vs_k$. The notation $V^{L_{\varphi}}$ for the space of invariants is to be understood in a scheme-theoretical way as the set of $v\in V$ such that for any $k$-algebra $R$, one has $\rho(x)v=v$ in $V\otimes_k R$ for all $x\in L_{\varphi}(R)$. In particular, if $(V,\rho)\in \Rep(P)$ and $\mathcal{V}(\rho)$ is the attached vector bundle on $\GZip^\mu$, the restriction of $\mathcal{V}(\rho)$ to $\Ucal_\mu$ is attached to the restriction of $\rho$ to $L_\varphi$, and the formula \eqref{muordsections} applies similarly.

By \eqref{globquot}, any $f\in V^{L_\varphi}=H^0(\Ucal_\mu,\mathcal{V}(\rho))$ corresponds bijectively to a unique function 
\begin{equation}\label{ftilde}
   \tilde{f} \colon U_\mu\to V 
\end{equation}
satisfying $\tilde{f}(1)=f$ and $\tilde{f}(axb^{-1})=\rho(a)\tilde{f}(x)$ for all $(a,b) \in E$ and all $x\in U_\mu$. The strategy to determine the space $H^0(\GZip^\mu,\mathcal{V}(\rho))$ will be to characterize which of these functions extend to a function $G_k \to V$. We will use Corollary \ref{corphial} for this purpose. As another preliminary, we introduce (a generalization of) the Brylinski--Kostant filtration in the next section.

\subsection{Brylinski--Kostant filtration}

\begin{lemma}\label{lem:phialpha}
Let $\alpha \in \Phi$. 
Let $V$ be a finite dimensional algebraic representation of $TU_{\alpha}$. 
Let $v \in V_{\nu}$ for $\nu \in X^*(T)$. 
Then we have 
\[
 u_{\alpha}(x)(v)-v=\sum_{j=1}^{\infty} x^j v_j 
\] 
where $v_j \in V_{\nu +j \alpha}$. 
\end{lemma}
\begin{proof}
This is proved in the proof of \cite[Proposition 3.3.2]{Donkin-good-filtrations-LNM}. We recall the argument. 
We write $u_{\alpha}(x)v$ as $\sum_{j \geq 0} x^j v_j$ for some $v_j \in V$. We note that $v_0=v$. 
By \eqref{eq:phiconj}, we have $v_j \in V_{\nu +j \alpha}$. 
\end{proof}

For $\alpha \in \Phi$, 
we define $E_{\alpha}^{(j)} \colon V \to V$ by 
\[
 u_{\alpha}(x)v=\sum_{j \geq 0} x^j E_{\alpha}^{(j)}(v)
\]
for $j \geq 0$ and 
put $E_{\alpha}^{(j)}=0$ if $j <0$. 
By Lemma \ref{lem:phialpha}, 
we have $E_{\alpha}^{(j)}(v) \in V_{\nu+j\alpha}$ for $v \in V_{\nu}$. 

Let $\Xi=(\alpha_1, \ldots, \alpha_m) \in \Phi^m$. 
Let $H$ be a closed subgroup scheme of 
$G$ contaning 
$T$ and $U_{\alpha_i}$ for $1 \leq i \leq m$. 
Let $V$ be a finite dimensional algebraic representation of $H$. 
Let 
$\mathbf{a}=(a_1,\ldots,a_m) \in (k^{\times})^m$ and 
$\mathbf{r}=(r_1,\ldots,r_m) \in \RR^m$. 
We put 
\begin{align*}
 (\ZZ^m)_{\mathbf{r}} &= \left\{ (n_1,\ldots,n_m) \in \ZZ^m \relmiddle| \sum_{i=1}^m n_i r_i =0 
 \right\}, \\ 
 \Lambda_{\Xi,\mathbf{r}} &= \left\{ \sum_{i=1}^m n_i \alpha_i \relmiddle| 
 (n_1,\ldots,n_m) \in (\ZZ^m)_{\mathbf{r}}
 \right\}. 
\end{align*}
For $[\nu] \in X^*(T)/\Lambda_{\Xi,\mathbf{r}}$, we put 
\[
 V_{[\nu]}=\bigoplus_{\nu \in [\nu]} V_{\nu} . 
\] 
We use the notation $\mathbf{j}$ for  $(j_1,\ldots,j_m) \in \ZZ^m$. 
For $[\mathbf{j}] \in \ZZ^m/(\ZZ^m)_{\mathbf{r}}$ 
and 
$[\nu] \in X^*(T)/\Lambda_{\Xi,\mathbf{r}}$, 
we put 
\begin{align*}
 [\mathbf{j}] \cdot \mathbf{r} 
 &= \sum_{i=1}^m j_i r_i \in \RR, \\ 
 [\nu] +[\mathbf{j}] \cdot \Xi 
 &= \left[ \nu +\sum_{i=1}^m j_i \alpha_i \right] \in 
 X^*(T)/\Lambda_{\Xi,\mathbf{r}}, 
\end{align*}
which are well-defined. 
For $[\nu] \in X^*(T)/\Lambda_{\Xi,\mathbf{r}}$ and a function $\delta \colon X^*(T) \to \RR$, we define 
$\fil_{\delta}^{\Xi,\mathbf{a},\mathbf{r}} V_{[\nu]}$ by 
\[
 \bigcap_{[\mathbf{j}] \in \ZZ^m/(\ZZ^m)_{\mathbf{r}}} 
 \bigcap_{\substack{\chi \in [\nu] +[\mathbf{j}] \cdot \Xi, \\ [\mathbf{j}] \cdot \mathbf{r} > \delta(\chi)}}
 \Ker \left( \sum_{\mathbf{j} \in [\mathbf{j}]} 
 \pr_{\chi} \circ  a_1^{j_1} E_{\alpha_1}^{(j_1)} \circ \cdots \circ a_m^{j_m} E_{\alpha_m}^{(j_m)} \colon V_{[\nu]
 } \to 
 V_{\chi} 
 \right) 
\]
where $\pr_{\chi} \colon V_{[\nu] +[\mathbf{j}] \cdot \Xi} \to V_{\chi}$ denotes the projection.

\begin{example}\label{example-risone}
Assume that $\Xi=(\alpha) \in \Phi$, $r_1=1$ and $\delta$ is a constant function $c \in \RR$. 
Then $\Lambda_{\Xi,\mathbf{r}}=0$ and 
$V_{[\nu]}=V_{\nu}$ for $\nu \in X^*(T)$. 
In this case, 
\begin{equation}\label{equ-Filc}
 \fil_{c}^{\Xi,\mathbf{a},\mathbf{r}} V_{\nu} = 
 \bigcap_{j > c} \Ker \left( 
 E_{\alpha}^{(j)} \colon V_{\nu} \to V_{\nu+j\alpha} \right),     
\end{equation}
which we simply write 
$\fil_{c}^{\alpha} V_{\nu}$. 
This is a Brylinski--Kostant filtration 
(\cf \cite[(3.3.2)]{Xiao-Zhu-on-vector-valued}). 
\end{example}

\subsection{Main result} \label{subsec-mainresult}
We now investigate the space of global sections over $\GZip^\mu$ of the vector bundle $\mathcal{V}(\rho)$ for $(V,\rho)\in \Rep(P)$. By \eqref{muordsections}, this space is contained in $V^{L_{\varphi}}$. Conversely, the problem is to determine which $f\in V^{L_{\varphi}}$ correspond to sections of $\mathcal{V}(\rho)$ that extend from $\Ucal_\mu$ to $\GZip^\mu$. Equivalently, we ask for which $f\in V^{L_{\varphi}}$ the regular function $\tilde{f} \colon U_\mu \to V$ defined in \eqref{ftilde} extends to a regular function $G_k \to V$.

Recall the definition of the integer $m_\alpha$ in \eqref{malpha-equ} for each $\alpha\in \Delta^P$. For example, if $P$ is defined over $\FF_q$, then $m_\alpha=1$ for all $\alpha\in \Delta^P$. We put $\mathbf{a}_\alpha=(-1,\ldots, -1) \in (k^{\times})^{m_\alpha}$. For $\alpha \in \Delta^P$, 
we put 
$\Xi_{\alpha}=(-\alpha, \sigma^{-1}(\alpha),\ldots,
 \sigma^{-(m_{\alpha}-1)}(\alpha))$
and 
$\mathbf{r}_{\alpha}=(r_{\alpha,1}, \ldots, r_{\alpha,m_{\alpha}})$, 
where $r_{\alpha,1}=1-\langle \alpha,\delta_{\alpha}  \rangle$ and 
\[
 r_{\alpha,i}=\frac{\langle \alpha,\delta_{\alpha}  \rangle -1}{q^{i-1}}  
\]
for $2 \leq i \leq m_{\alpha}$. 
We view $\delta_{\alpha}$ as 
a function $X^*(T) \to \RR$ 
by $\chi \mapsto \langle \chi ,\delta_{\alpha} \rangle$. 

\begin{theorem}\label{thm-main}
Let $(V,\rho)\in \Rep(P)$. Via the inclusion $H^0(\GZip^\mu,\mathcal{V}(\rho))\subset V^{L_{\varphi}}$ (see Corollary \ref{cor-global-muord}) one has an identification
\begin{equation}\label{eq-main}
H^0(\GZip^\mu,\mathcal{V}(\rho))=V^{L_{\varphi}}\cap 
 \bigcap_{\alpha \in \Delta^P} 
 \bigoplus_{[\nu] \in X^*(T)/\Lambda_{\Xi_{\alpha},\mathbf{r}_{\alpha}}} 
 \fil_{\delta_{\alpha}}^{\Xi_{\alpha},\mathbf{a}_{\alpha},\mathbf{r}_{\alpha}} V_{[\nu]} .
\end{equation}
\end{theorem}

\begin{proof}
Let $f\in V^{L_{\varphi}}$, and let $\tilde{f} \colon U_\mu\to V$ be the function defined in \eqref{ftilde}. It suffices to show: $\tilde{f}$ extends to $G$ if and only if 
\[
 f \in 
  \bigoplus_{[\nu] \in X^*(T)/\Lambda_{\Xi_{\alpha},\mathbf{r}_{\alpha}}} 
 \fil_{\delta_{\alpha}}^{\Xi_{\alpha},\mathbf{a}_{\alpha},\mathbf{r}_{\alpha}} V_{[\nu]}
\] 
for all $\alpha \in \Delta^P$. By Corollary \ref{corphial}, $\tilde{f}$ extends to $G_k$ if and only if $\tilde{f}\circ \psi_\alpha  \colon  Y_0\to V$ extends to a function $Y\to V$. 
We now give an explcit formula for $\tilde{f}\circ \psi_\alpha ((x,y),t)$. Using \eqref{eq:phiadec} and \eqref{eq:inE}, the element $\psi_\alpha ((x,y),t)\in U$ can be written as $x_1 x_2^{-1}$ with 
$(x_1,x_2)\in E$ and
\begin{align*}
 x_1  = x 
      \phi_{\alpha} \left(\begin{pmatrix}
1&0\\-t^{-1}&1
\end{pmatrix} \right) \delta_{\alpha}(t) 
u_{t,\alpha} , \quad 
 x_2=y 
\varphi(\delta_{\alpha}(t))
 \phi_{\alpha} \left(\begin{pmatrix}
1& -t_{\alpha} \\ 0&1
\end{pmatrix} \right) 
u_{t,\alpha}. 
\end{align*}
It follows:
\begin{align*}
    (\tilde{f}\circ \psi_\alpha) ((x,y),t)
    =\tilde{f}(x_1 x_2^{-1})=\rho(x_1)f 
     = \rho(x) \rho \left(       \phi_{\alpha} \left(\begin{pmatrix}
1&0\\-t^{-1}&1
\end{pmatrix} \right) 
\delta_{\alpha}(t) 
u_{t,\alpha} 
\right) f. 
\end{align*}
Hence, the function $\tilde{f}\circ \psi_\alpha$ extends to $Y$ if and only if the function 
\[
 F_\alpha \colon t\mapsto \rho \left(       \phi_{\alpha} \left(\begin{pmatrix}
1&0\\-t^{-1}&1
\end{pmatrix} \right) 
\delta_{\alpha}(t) 
u_{t,\alpha} 
\right) f
\] 
lies in $k[t]\otimes V$. 
Write $f=\sum_{\nu \in X^*(T)} f_{\nu}$ by the weight decomposition of $f$. 
We put 
\begin{equation}
f_{\nu,\Xi_{\alpha}}^{\mathbf{j}} = E_{-\alpha}^{(j_1)} E_{\sigma^{-1}(\alpha)}^{(j_2)}
\cdots
E_{\sigma^{-(m_{\alpha}-1)}(\alpha)}^{(j_{m_{\alpha}})} f_{\nu} \in V_{\nu +\mathbf{j} \cdot \Xi_{\alpha}} 
\end{equation}
for $\mathbf{j}=(j_1,\ldots,j_{m_{\alpha}}) \in \ZZ^{m_{\alpha}}$ and $\nu \in X^*(T)$. 
We obtain 
\begin{align*}
 F_\alpha(t) &= 
 \rho \left( \delta_{\alpha}(t) 
 \phi_{\alpha} \left(\begin{pmatrix}
1&0\\-\alpha(\delta_{\alpha}(t)) t^{-1}&1
\end{pmatrix} \right)  
u_{t,\alpha} 
\right) f \\ 
 &= 
 \sum_{\nu} 
 \rho \left( \delta_{\alpha}(t) 
 \phi_{\alpha} \left(\begin{pmatrix}
1&0\\- t^{\langle \alpha, \delta_{\alpha} \rangle-1}&1
\end{pmatrix} \right)  
\prod_{i=2}^{m_{\alpha}} \phi_{\sigma^{-(i-1)}(\alpha)} 
\left(\begin{pmatrix}
1&- t_{\alpha}^{\frac{1}{q^{i-1}}}\\0&1
\end{pmatrix} \right) 
\right)
f_{\nu}\\
&=
 \sum_{\nu} 
 \rho \left( \delta_{\alpha}(t) \right) 
 \sum_{\mathbf{j} \in \ZZ^{m_{\alpha}}} 
 \left( 
 (-t^{\langle \alpha, \delta_{\alpha} \rangle-1})^{j_1}  
 \prod_{i=2}^{m_{\alpha}}  (-t_{\alpha}^{\frac{1}{q^{i-1}}})^{j_i} \right) 
 f_{\nu,\Xi_{\alpha}}^{\mathbf{j}}  \\ 
 &=
  \sum_{\nu} 
 \sum_{\mathbf{j} \in \ZZ^{m_{\alpha}}} 
  t^{\langle \nu +\mathbf{j} \cdot \Xi_{\alpha} ,  \delta_{\alpha}  \rangle} 
 \left( 
 (-t^{\langle \alpha, \delta_{\alpha} \rangle-1})^{j_1}  
 \prod_{i=2}^{m_{\alpha}}  (-t_{\alpha}^{\frac{1}{q^{i-1}}})^{j_i} \right) 
 f_{\nu,\Xi_{\alpha}}^{\mathbf{j}} . 
\end{align*}
For fixed $\chi \in X^*(T)$, let 
$F_{\alpha,\chi}(t)$ be 
the $V_{\chi}$-component of $F_\alpha(t)$. 
Then we have 
\begin{align*}
 F_{\alpha,\chi}(t) &= 
 \sum_{\mathbf{j} \in \ZZ^{m_{\alpha}}} 
 t^{\langle \chi , \delta_{\alpha}  \rangle} 
 \left( 
 (-t^{\langle \alpha , \delta_{\alpha}  \rangle-1})^{j_1}  
 \prod_{i=2}^{m_{\alpha}} (- t_{\alpha}^{\frac{1}{q^{i-1}}})^{j_i} 
 \right) 
 f_{\chi -\mathbf{j} \cdot \Xi_{\alpha},\Xi_{\alpha}}^{\mathbf{j}} \\ 
 &=\sum_{[\mathbf{j}] \in \ZZ^{m_{\alpha}}/(\ZZ^{m_{\alpha}})_{\mathbf{r}_{\alpha}}} 
 \sum_{\mathbf{j} \in [\mathbf{j}]} 
 t^{\langle \chi , \delta_{\alpha}  \rangle -\mathbf{j} \cdot \mathbf{r}_{\alpha}} (-1)^{\sum_{i=1}^{m_{\alpha}} j_i} 
 f_{\chi -\mathbf{j} \cdot \Xi_{\alpha}
,\Xi_{\alpha}}^{\mathbf{j}} . 
\end{align*}
The exponents of $t$ in two terms 
in the last expression are equal if and only if the indices belong to the same coset in 
$\ZZ^{m_{\alpha}}/(\ZZ^{m_{\alpha}})_{\mathbf{r}_{\alpha}}$. 
Therefore, 
$F_{\alpha,\chi}(t)$ lies in $k[t]\otimes V_{\chi}$ for all $\chi \in X^*(T)$ 
if and only if we have 
\[
 \sum_{\mathbf{j} \in [\mathbf{j}]} 
(-1)^{\sum_{i=1}^{m_{\alpha}} j_i} 
 f_{\chi -\mathbf{j} \cdot \Xi_{\alpha},\Xi_{\alpha}}^{\mathbf{j}} 
 =0  
\]
for all $\chi \in X^*(T)$ and 
$[\mathbf{j}] \in \ZZ^{m_{\alpha}}/(\ZZ^{m_{\alpha}})_{\mathbf{r}_{\alpha}}$ such that 
$\mathbf{j} \cdot \mathbf{r}_{\alpha} > \langle \chi , \delta_{\alpha}  \rangle$. 
This condition is equivalent to that 
$f$ belongs to 
$  \bigoplus_{[\nu] \in X^*(T)/\Lambda_{\Xi_{\alpha},\mathbf{r}_{\alpha}}} 
 \fil_{\delta_{\alpha}}^{\Xi_{\alpha},\mathbf{a}_{\alpha},\mathbf{r}_{\alpha}} V_{[\nu]}$. 
Hence the claim follows. 
\end{proof}

We now give some corollaries of Theorem \ref{thm-main} in case where the formula \eqref{eq-main} becomes simpler. For $\nu \in X^*(T)$ and 
$\chi \in X^*(T)_{\RR}$, 
we put 
\begin{equation}\label{equ-FilPchi}
 \fil_{\chi}^P 
 V_{\nu} = 
 \bigcap_{\alpha \in \Delta^P} 
 \fil_{\langle \chi, \alpha^\vee \rangle}^{-\alpha} 
 V_{\nu}    
\end{equation}
where $\fil_{\langle \chi, \alpha^\vee \rangle}^{-\alpha} V_{\nu}$ was defined in Example \ref{example-risone}. The morphism $\wp \colon T \to T$ induces the isomorphism 
\begin{equation}\label{equ-Pupstar}
 \wp^* \colon X^*(T)_{\RR} \stackrel{\sim}{\longrightarrow} X^*(T)_{\RR};\ \lambda \mapsto \lambda \circ \wp  = \lambda -q\sigma^{-1} (\lambda).    
\end{equation}

\begin{corollary}\label{cor-main-Fq}
Assume that $P$ is defined over $\FF_q$. 
Let $(V,\rho)\in \Rep(P)$. Via the inclusion $H^0(\GZip^\mu,\mathcal{V}(\rho))\subset V^{L(\FF_q)}$ one has
\begin{equation*}
H^0(\GZip^\mu,\mathcal{V}(\rho))=V^{L(\FF_q)}\cap 
 \bigoplus_{\nu \in X^*(T)} 
 \fil_{{\wp^*}^{-1}(\nu)}^P 
 V_{\nu}.
\end{equation*}
\end{corollary}
\begin{proof}
For $\alpha \in \Delta^P$ and 
$\nu \in X^*(T)$, 
we have 
\[
 \fil_{\delta_{\alpha}}^{\Xi_{\alpha},\mathbf{a}_{\alpha},\mathbf{r}_{\alpha}} V_{[\nu]} = 
 \fil_{\langle \nu,\delta_{\alpha} \rangle}^{-\alpha} V_{\nu} 
 = \fil_{\langle {\wp^*}^{-1}(\nu),\alpha^{\vee}  \rangle}^{-\alpha} V_{\nu} . 
\]
Hence the claim follows from 
Lemma \ref{lemLphi}\eqref{lemLphi-item3} and  Theorem \ref{thm-main}. 
\end{proof}

Assume again that $P$ is defined over $\FF_q$. To simplify further, assume that $(V,\rho)\in \Rep(P)$ is trivial on the unipotent radical $R_{\mathrm{u}}(P)$. Then we have $E^{(j)}_{-\alpha}=0$ for all $\alpha\in \Delta^P$ and all $j>0$. It follows that $\fil_{c}^{-\alpha} V_{\nu}=V_\nu$ for $c\geq 0$ and $\fil_{c}^{-\alpha} V_{\nu}=0$ for $c<0$. We obtain that for all $\chi\in X^*(T)_\RR$, one has
\begin{equation}
\fil_{\chi}^{P} V_{\nu}=
\begin{cases}
V_{\nu} & \textrm{ if for all }\alpha\in \Delta^P \textrm{ one has} \ \langle \chi, \alpha^\vee \rangle \geq 0,  \\
0 & \textrm{ otherwise.} 
\end{cases}
\end{equation}
Define a subspace $V_{\geq 0}^{\Delta^P}\subset V$ as follows: 
\begin{equation}\label{equ-VDeltaP}
V_{\geq 0}^{\Delta^P} = \bigoplus_{\substack{\langle \nu,\delta_{\alpha} \rangle \geq 0, \ 
 \forall \alpha\in \Delta^P}} V_\nu.
\end{equation}
For example, if $T$ is split over $\FF_q$, then 
$\delta_{\alpha}=-\alpha^\vee /(q-1)$, 
and therefore $V_{\geq 0}^{\Delta^P}$ is the direct sum of the weight spaces $V_\nu$ for those $\nu\in X^*(T)$ satisfying $\langle \nu,\alpha^\vee \rangle \leq 0$ for all $\alpha \in \Delta^P$.

\begin{corollary}\label{cor-Fq-Levi}
Assume that $P$ is defined over $\FF_q$ and furthermore that $(V,\rho)\in \Rep(P)$ is trivial on the unipotent radical $R_{\mathrm{u}}(P)$. Then one has an equality
\begin{equation}
H^0(\GZip^\mu,\mathcal{V}(\rho))=V^{L(\FF_q)}\cap V_{\geq 0}^{\Delta^P} .
\end{equation}
\end{corollary}
This formula recovers the result \cite[Theorem 1]{Koskivirta-automforms-GZip} (with slightly different notation). In \loccitn, only the special case when $G$ is split over $\FF_p$ and $V$ is of the form $V_I (\lambda)$ was considered.

\subsection{Perfection}\label{subsec-perfect}
As noted in Remark \ref{rem-shtukas}, the perfection of the stack of $G$-zips appears in connection with the moduli of local shtukas. In \cite[Lemma 5.3.6]{Xiao-Zhu-geometric-satake}, the zip datum that appears satisfes that $P$ is defined over $\FF_q$. We do not make this assumption here. For a scheme $X$ over $k$, define the perfection of $X$ as the projective limit 
\[
 X^{\pf}:=\varprojlim_{\varphi_X} X 
\]
where $\varphi_X$ denotes the absolute $q$-th power Frobenius endomorphism of $X$. 
There is a natural map $X^{\pf}\to X$. 
We have an isomorphism 
\[
 X^{\pf} \simeq \varprojlim 
 \left( \cdots \xrightarrow{\varphi} X^{(q^{-2})} \xrightarrow{\varphi} X^{(q^{-1})} \xrightarrow{\varphi} X \right) 
\]
where $\varphi$ denotes the relative $q$-th power Frobenius endomorphism. 
The perfection of $\GZip^\mu$ is then given by
\begin{equation}
\GZip^{\mu,\pf}=[E^{\pf}\backslash G_k^{\pf}].
\end{equation}
Similarly to Proposition \ref{prop-Umu}, the perfection of the $\mu$-ordinary locus $\Ucal_\mu^{\pf}$ is isomorphic to $[1/L_{\varphi}^{\pf}]$. Since $L_{\varphi}=L_{\varphi}^\circ\rtimes L_0(\FF_q)$ by Lemma \ref{lemLphi}\eqref{lemLphi-item2}, we obtain
\begin{equation}\label{equ-Umupf}
    \Ucal_\mu^{\pf}=[1/L_0(\FF_q)].
\end{equation}

If $(V,\rho)$ is a $P$-representation, then we obtain a $P^{\pf}$-representation by pull-back, which we denote by $\rho^{\pf}$. 
This yields a vector bundle $\mathcal{V}(\rho^{\pf})$ on $\GZip^{\mu,\pf}$, which also coincides with the pull-back of $\mathcal{V}(\rho)$ under the natural map $\GZip^{\mu,\pf}\to \GZip^\mu$. By the equation \eqref{equ-Umupf} above, we see that the space $H^0(\GZip^{\mu,\pf},\mathcal{V}(\rho^{\pf}))$ is naturally a subspace of $V^{L_0(\FF_q)}$.
\begin{corollary}\label{cor-perf}
Let $(V,\rho)\in \Rep(P)$. We have
\begin{equation}\label{eq-main-perf}
H^0(\GZip^{\mu,\pf},\mathcal{V}(\rho^{\pf}))=V^{L_{0}(\FF_q)}\cap 
 \bigcap_{\alpha \in \Delta^P} 
 \bigoplus_{[\nu] \in X^*(T)/\Lambda_{\Xi_{\alpha},\mathbf{r}_{\alpha}}} 
 \fil_{\delta_{\alpha}}^{\Xi_{\alpha},\mathbf{a}_{\alpha},\mathbf{r}_{\alpha}} V_{[\nu]} .
\end{equation}
\end{corollary}
\begin{proof}
Let $d$ be the smallest positive integer such that $\mu$ is defined over $\FF_{q^d}$. 
We show that $H^0(\GZip^{\mu,\pf},\mathcal{V}(\rho^{\pf}))$ is given by the subspace of elements $f\in V$ such that there exists $n\geq 1$ with $f\in H^0(\GZip^{\mu},\mathcal{V}(\rho^{(q^{nd})}))$. Indeed, such a section is given by a map $f\colon G_k^{\pf}\to V$ satisfying an $E^{\pf}$-equivariance condition with respect to $\rho^{\pf}$. Since $V$ is a scheme of finite-type, such a map is given by a map $f_{n}\colon G_k\to V$ at a finite level of the system $(\cdots \xrightarrow{\varphi^d} G_k \xrightarrow{\varphi^d} G_k)$. We have
\begin{equation}
\fil_{\delta_{\alpha}}^{\Xi_{\alpha},\mathbf{a}_{\alpha},\mathbf{r}_{\alpha}} V^{(q^{nd})}_{[q^{nd}\nu]} = \fil_{\delta_{\alpha}}^{\Xi_{\alpha},\mathbf{a}_{\alpha},\mathbf{r}_{\alpha}} V_{[\nu]}.
\end{equation}
Hence, changing $\rho$ to $\rho^{(q^n)}$ only affects $V^{L_\varphi}$. The result follows.
\end{proof}

\subsection{$L$-semisimplification}
If $\rho \colon P\to \GL(V)$ is an arbitrary representation, we can attach a $P$-representation $(V,\rho^{L\textrm{-ss}})$ which is trivial on $R_{\mathrm{u}}(P)$. The representation $\rho^{L\textrm{-ss}}$ is defined as the composition
\begin{equation}
    \rho^{L\textrm{-ss}}  \colon  P \xrightarrow{\theta^P_L} L \xrightarrow{\rho} \GL(V)
\end{equation}
where $\theta^P_L \colon P\to L$ is the natural projection map whose kernel is $R_{\mathrm{u}}(P)$, as defined in \S\ref{subsec-zipdatum}. We call $\rho^{L\textrm{-ss}}$ the $L$-semisimplification of $\rho$. We sometimes write $V^{L\textrm{-ss}}$ to denote this representation (even though the underlying vector space is the same as $V$).

One obvious property of $V^{L\textrm{-ss}}$ is $(V^{L\textrm{-ss}})^{L_\varphi}=V^{L_\varphi}$ since $L_\varphi\subset L$ by Lemma \ref{lemLphi}\eqref{lemLphi-item1}. In particular, by Corollary \ref{cor-global-muord}, we have for all $(V,\rho)\in \Rep(P)$, the equality
\begin{equation}\label{identULss}
    H^0(\Ucal_\mu,\mathcal{V}(\rho^{L\textrm{-ss}})) = H^0(\Ucal_\mu,\mathcal{V}(\rho)).
\end{equation}
Note that this identification is somewhat indirect: it is not induced by a morphism between the sheaves $\mathcal{V}(\rho)$ and $\mathcal{V}(\rho^{L\textrm{-ss}})$. For $f\in H^0(\Ucal_\mu,\mathcal{V}(\rho))$, we will write $f^{L\textrm{-ss}}$ for its image under the identification \eqref{identULss}, and call it the $L$-semisimplification of $f$. As an element of $V$, $f^{L\textrm{-ss}}$ is the same as $f$, but we want to emphasize the fact that the representation has changed.

We now give another interpretation of $L$-semisimplification when $P$ is defined over $\FF_q$. Write again $U_\mu \subset G_k$ for the unique open $E$-orbit, and recall that $1\in U_{\mu}$ (see \S\ref{subsec-ordloc}).
\begin{lemma}
Assume that $P$ is defined over $\FF_q$. There exists a unique regular map $\Theta \colon U_\mu\to L$ such that for any $(a,b)\in E$, one has
\begin{equation}\label{Theta-equ}
    \Theta(ab^{-1})=\theta^P_L(a) \theta^Q_L(b)^{-1}.
\end{equation}
Furthermore, we have $L\subset U_\mu$ and 
the inclusion $L\subset U_\mu$ is a section of $\Theta$. 
\end{lemma}

\begin{proof}
First, note that since $P$ is defined over $\FF_q$, one has $L=M$, hence the formula \eqref{Theta-equ} makes sense. The unicity of $\Theta$ is obvious. 
For the existence, consider the map $\tilde{\Theta} \colon E\to L;\ (a,b)\mapsto \theta^P_L(a) \theta^Q_L(b)^{-1}$. Since $P$ is defined over $\FF_q$, one has $L_\varphi=L(\FF_q)$ (Lemma \ref{lemLphi}\ref{lemLphi-item3}). For all $(a,b)\in E$ and all $x\in L(\FF_q)$, one has $\tilde{\Theta}(ax,bx)=\tilde{\Theta}(a,b)$. Hence $\tilde{\Theta}$ factors to a map $\Theta \colon E/L(\FF_q)\simeq U_\mu\to L$. This proves the first result. Now, if $x\in L$, we can write $x=a\varphi(a)^{-1}$ with $a\in L$ by Lang's theorem. 
Hence $x \in U_\mu$ and $\Theta(x)=a\varphi(a)^{-1}=x$, so the second statement is proved.
\end{proof}

\begin{example}
Consider the case $G=\Sp(2n)_{\FF_q}$ for $n\geq 1$. 
We write an element of $G_k$ as 
\[
 \begin{pmatrix} A&B\\C&D  \end{pmatrix} 
\]
with $A,B,C,D$ square matrices of size $n\times n$. Let $P\subset G_k$ be the parabolic subgroup defined by the condition $B=0$ and $Q\subset G_k$ the parabolic subgroup defined by the condition $C=0$. We put $L = P\cap Q$. This gives a zip datum $(G,P,L,Q,L,\varphi)$. The Zariski open subset $U_\mu\subset G_k$ is the set of matrices in $G_k$ for which $A$ is invertible. The map $\Theta \colon U_\mu\to L$ is given by
\begin{equation}
    \Theta \colon \left(\begin{matrix} A&B\\C&D  \end{matrix} \right) \mapsto \left(\begin{matrix} A&0\\0&D-CA^{-1}B  \end{matrix} \right).
\end{equation}
\end{example}

\begin{proposition}
Assume that $P$ is defined over $\FF_q$. Let $(V,\rho)\in \Rep(P)$ and let $f\in V^{L(\FF_q)}$. Let $\tilde{f}$ be the corresponding function $U_\mu\to V$ defined in \eqref{ftilde}. Then the function $\widetilde{f^{L\textrm{-ss}}} \colon U_\mu\to V$ that corresponds to the $L$-semisimplification $f^{L\textrm{-ss}}$ is the composition
\begin{equation}
    U_{\mu} \stackrel{\Theta}{\longrightarrow} L 
    \hookrightarrow U_\mu \stackrel{\tilde{f}}{\longrightarrow} V. 
\end{equation}
\end{proposition}
\begin{proof}
Put $f' = \tilde{f}\circ \Theta$. For $(a,b)\in E$ and $g\in U_\mu$ such that $g=ab^{-1}$, we have
\begin{equation}
f'(g)=f'(ab^{-1})=\tilde{f}(\Theta(ab^{-1}))=\tilde{f}(\theta^P_L(a)\theta^Q_L(b)^{-1})=\rho(\theta^P_L(a))f=\rho^{L\textrm{-ss}}(a)f=\widetilde{f^{L\textrm{-ss}}}(g).    
\end{equation}
Hence $f'=\widetilde{f^{L\textrm{-ss}}}$.
\end{proof}

Let $f\in H^0(\GZip^\mu,\mathcal{V}(\rho))$ be a global section. We may view its restriction $f|_{\Ucal_\mu}$ as a section of $\mathcal{V}(\rho^{L\textrm{-ss}})$ over $\Ucal_\mu$ by the identification \eqref{identULss}. It is thus natural to ask if $(f|_{\Ucal_\mu})^{L \textrm{-ss}}$ extends to a global section over $\GZip^\mu$. We prove that this holds when $P$ is defined over $\FF_q$ in the following proposition.

\begin{proposition} \label{Lssreg}
Assume that $P$ is defined over $\FF_q$. The identification \eqref{identULss} extends to a commutative diagram
\begin{equation}
\xymatrix@1@M=7pt{
H^0(\GZip^{\mu},\mathcal{V}(\rho)) \ar@{^{(}->}[r]  \ar@{^{(}->}[d] & H^0(\GZip^\mu,\mathcal{V}(\rho^{L\textrm{-ss}}))  \ar@{^{(}->}[d] \\
H^0(\Ucal_\mu,\mathcal{V}(\rho)) \ar[r]^-{=} &H^0(\Ucal_\mu,\mathcal{V}(\rho^{L\textrm{-ss}})). }
\end{equation}
\end{proposition}
\begin{proof}
Let $f\in H^0(\GZip^{\mu},\mathcal{V}(\rho))$. Since $P$ is defined over $\FF_q$, we can apply Corollary \ref{cor-main-Fq} to the representation $(V,\rho)$. Furthermore, since $R_{\mathrm{u}}(P)$ acts trivially on $(V^{L\textrm{-ss}},\rho^{L\textrm{-ss}})$, we can apply Corollary \ref{cor-Fq-Levi} to $(V^{L\textrm{-ss}},\rho^{L\textrm{-ss}})$. Therefore, it suffices to show that
for each $\nu\in X^*(T)$, 
\begin{equation}
    V^{L(\FF_q)}\cap 
 \bigoplus_{\nu \in X^*(T)} 
 \fil_{{\wp^*}^{-1}(\nu)}^P 
 V_{\nu} \quad  
 \subset 
 \quad 
 V^{L(\FF_q)}\cap 
V_{\geq 0}^{\Delta^P}.
\end{equation}
By \eqref{equ-VDeltaP}, it suffices to show the following: for any fixed $\nu\in X^*(T)$, if $\fil_{{\wp^*}^{-1}(\nu)}^P 
 V_{\nu}\neq 0$, then $\langle {\wp^*}^ {-1}(\nu),\alpha^\vee \rangle \geq 0$ 
for all $\alpha\in \Delta^P$. More generally, using \eqref{equ-FilPchi}, it suffices to show that for any $\alpha\in \Delta^P$ and any integer $c\in \ZZ$ such that $\fil_{c}^{-\alpha} V_{\nu}\neq 0$, one has $c\geq 0$. This is trivial by \eqref{equ-Filc} because 
$E^{(0)}_{-\alpha}$ is the identity map. 
\end{proof}

\begin{rmk}
Proposition \ref{Lssreg} does not hold in general without the assumption that $P$ is defined over $\FF_q$ 
as an example in \S \ref{counter-Lss} shows. 
\end{rmk}

\section{The case of $G=\SL_{2,\FF_q}$} \label{sec-SL2-proof}

\subsection{Notation for $\SL_2$}\label{NotaSL2-sec}
Let $B_2$ and $B_2^+$ be 
the lower-triangular and upper-triangular
Borel subgroup of $\SL_{2,k}$. 
Let $T_2$ be the diagonal torus of $\SL_{2,k}$. 
We put 
\[
 u_2 = \begin{pmatrix} 1 & 0 \\ 1 & 1 \end{pmatrix} \in B_2(k) . 
\]
For $r \in \ZZ$, 
let $\chi_r$ be the character of $B_2$ defined by 
\[
  \begin{pmatrix} x & 0 \\ z & x^{-1} \end{pmatrix}  \mapsto x^r. 
\] 
Let $\Std \colon \SL_{2,k}\to \GL_{2,k}$ be the standard representation. 
Restrictions of $\chi_r$ and $\Std$ to subgroups 
are denoted by the same notations. 

\subsection{Zip datum}\label{SL2zip-sec}
Let $G=\SL_{2,\FF_q}$ and $\mu \colon \GG_{\mathrm{m},k} \to G_k;\ x\mapsto \diag(x,x^{-1})$. 
Let $\Zcal_{\mu}=(G,P,L,Q,M,\varphi)$ be the associated zip datum. 
We have 
$P=B_2$, $Q=B_2^+$ and 
$L=M=T_2$. 
We take $(B,T)=(B_2,T_2)$ as a Borel pair and 
take a frame as in Lemma \ref{lem-framemu}. 
Denote by $\alpha$ the unique element of $\Delta$. In our convention of positivity, $\alpha=\chi_{2}$. Note that $I=\emptyset$ and $\Delta^P=\{\alpha\}$. Identify $X^*(T)=\ZZ$ such that $r\in \ZZ$ corresponds to the character $\chi_r$. 
The zip group $E$ is equal to 
\begin{equation}\label{xy}
    \left\{ \left( \begin{pmatrix} a&0\\c&a^{-1} \end{pmatrix},  \begin{pmatrix} a^q&b\\0&a^{-q} \end{pmatrix}  \right)\in B_2 \times B_2^+ \right\}.
\end{equation}
The unique open $E$-orbit $U_\mu \subset G_k$  
is given by 
\begin{equation}
    U_\mu=\left\{ \left(\begin{matrix}
    x & y \\ z & w 
    \end{matrix}
    \right)\in \SL_{2,k} \relmiddle| x \neq 0 \right\}.
\end{equation}

\subsection{The space $H^0(\GZip^\mu,\mathcal{V}(\rho))$}
Let $\rho \colon B\to \GL(V)$ be a representation. We write the weight decomposition as $V=\bigoplus_{i\in \ZZ}V_i$ where $T$ acts on $V_i$ by the character $\chi_i$ for all $i\in \ZZ$. We have 
\[
 H^0(\Ucal_\mu,\mathcal{V}(\rho))=V^{L(\FF_q)}
 =\bigoplus_{i\in (q-1)\ZZ} V_i 
\] 
by Corollary \ref{cor-global-muord}. 
Since in this case the parabolic $P=B$ is defined over $\FF_q$, we can apply Corollary \ref{cor-main-Fq} to compute the space of global section  $H^0(\GZip^\mu,\mathcal{V}(\rho))$. Also, since $T$ is split over $\FF_q$, the map $\wp^*$ is given by $\nu \mapsto -(q-1)\nu$, hence ${\wp^*}^{-1}(\nu)=\frac{-\nu}{q-1}$. 
We obtain 
\[
H^0(\GZip^\mu,\mathcal{V}(\rho)) =V^{L(\FF_q)}\cap 
 \bigoplus_{i\in \ZZ} 
 \fil_{\frac{-\chi_i}{q-1}}^P 
 V_i\label{equ-H0-SL2-line1} 
 = \bigoplus_{i\in -(q-1)\NN} 
 \fil_{\frac{-i}{q-1}}^{-\alpha}
 V_{i}, 
\]
where we used that $\fil_{\frac{-i}{q-1}}^{-\alpha}V_i=0$ for $i>0$. In particular, $H^0(\GZip^\mu,\mathcal{V}(\rho))$ is stable by $T$ and is entirely determined by its weight spaces $\fil_{\frac{-i}{q-1}}^{-\alpha} V_{i} \subset V_i$ for $i\in -(q-1)\NN$. 
Let $(V,\rho)\in \Rep(B)$ and set $n = \dim(V)$. Set $V_{\leq i}=\bigoplus_{j\leq i}V_j$ and $V_{\geq i}=\bigoplus_{j\geq i}V_j$. Then using Lemma \ref{lem:phialpha}, we have a $B$-stable filtration
\begin{equation}
  \cdots \subset  V_{\leq {i-1}}\subset V_{\leq i} \subset V_{\leq i+1}\subset \cdots . 
\end{equation}
For all $i\in -(q-1)\NN$, we have 
\begin{equation}\label{eqf1}
H^0(\GZip^\mu,\mathcal{V}(\rho))_i = \left\{ f \in V_i \relmiddle| \rho ( u_2 ) f \in V_{\geq \frac{(q+1)i}{q-1}} 
\right\}
\end{equation}
by the definition of $\fil_{\frac{-i}{q-1}}^{-\alpha}
 V_{i}$. 

\begin{lemma}
Let $(V,\rho)\in \Rep(B)$ and $m\in \ZZ$ be the smallest weight of $\rho$. Then one has an inclusion
\begin{equation}\label{eqinclu}
  \bigoplus_{\substack{i\in -(q-1)\NN, \\ (q+1)i\leq (q-1)m}} V_i \subset H^0(\GZip^{\mu},\mathcal{V}(\rho)).
\end{equation}
\end{lemma}

\begin{proof}
Let $f\in V_i$ with $i\in -(q-1)\NN$ and $(q+1)i\leq (q-1)m$. Then we have $V_{\geq \frac{(q+1)i}{q-1}}=V$, 
so we have $f\in H^0(\GZip^\mu,\mathcal{V}(\rho))_i$.
\end{proof}

The following example shows that  $H^0(\GZip^\mu,\mathcal{V}(\rho))$ is not a sum of weight spaces of $V$ in general. 

\begin{example}\label{example-notsum}
For $i\in \{1,-1\}$, let $e_i$ be a nonzero vector of weight $i$ of $\Std$. Consider $\rho:=\Std\otimes \Std$ with basis $e_i\otimes e_j$ for $i,j\in \{1,-1\}$. 
The weights of $\rho$ are $\{2,0,-2\}$, and $\dim(V_2)=\dim(V_{-2})=1$, $\dim(V_0)=2$. Then we have
\[
 H^0(\GZip^\mu,\mathcal{V}(\rho))_0 = \Span(e_1\otimes e_{-1} - e_{-1}\otimes e_1).
\]

\end{example}

\subsection{Property (P)}\label{subsec-PropP}

\begin{proposition}\label{propPequiv}
    Let $\rho  \colon  B\to \GL(V)$ be an algebraic representation. Let $m_1, \ldots, m_n$ be the weights of $V$ ordered so that $m_1 > m_2 > \cdots > m_n$. 
    The following properties are equivalent.
\begin{equivlist}
\item \label{Pitem1} The subspace $V^{R_{\mathrm{u}}(B)}$ is one-dimensional (and hence is equal to $V_{m_n}$).
\item \label{Pitemint} The intersection of all nonzero $B$-subrepresentations in $V$ is nonzero.
\item \label{Pitempnz} For all $1\leq i \leq n$, we have $\dim(V_{m_i})= 1$ and for any $v\in V_{m_i}\setminus \{0\}$, the projection of $\rho (u_2)v$ onto $V_{m_n}$ is nonzero.
\end{equivlist}
\end{proposition}
\begin{proof}
We show (i) $\Rightarrow$ (ii). If $W\subset V$ is a nonzero $B$-subrepresentation, then $W^{R_{\mathrm{u}}(B)}\subset V^{R_{\mathrm{u}}(B)}$. Since $W^{R_{\mathrm{u}}(B)}\neq 0$, we have $W^{R_{\mathrm{u}}(B)}=V^{R_{\mathrm{u}}(B)}$ 
hence $V^{R_{\mathrm{u}}(B)}\subset W$. 

We show (ii) $\Rightarrow$ (iii). 
We show that for any nonzero $v\in V_{m_i}$ the projection of 
$\rho (u_2) v$ onto $V_{m_n}$ is nonzero. 
For a contradiction, assume it is zero. 
Since $B=R_{\mathrm{u}}(B)T$, 
the $B$-subrepresentation generated by $v$ 
is generated by $v$ as an $R_{\mathrm{u}}(B)$-representation. 
Hence this representation does not have a non-trivial 
intersection with $V_{m_n}$ by Lemma \ref{lem:phialpha}. 
This contradicts (ii). Hence the claim follows. 
We note that $\dim V_{m_n} =1$ by (ii). 
Assume that $\dim V_{m_i} \geq 2$ for some $i$. 
Then there is a nonzero $v \in V_{m_i}$ such that the projection to $V_{m_n}$ of  
$\rho (u_2 )v$ is zero. 
This is a contradiction. 

We show (iii) $\Rightarrow$ (i). 
Assume $\dim V^{R_{\mathrm{u}}(B)} \geq 2$. 
Then $V^{R_{\mathrm{u}}(B)}$ contains $V_{m_i}$ for some $i \neq n$. 
For any nonzero $v \in V_{m_i} \subset V^{R_{\mathrm{u}}(B)}$, 
the projection of $\rho (u_2)v$ onto $V_{m_n}$ is zero. 
This is a contradiction. 
\end{proof}

We say that 
$(V,\rho)\in \Rep(B)$ satisfies the property (P) 
if the equivalent conditions of 
Proposition \ref{propPequiv} are satisfied. 

\begin{example}\label{egIndP}
For $\lambda\in X^*_{+}(T)$, the restriction to $B$ of $\Ind_B^{G_k}(\lambda)$ satisfies the property (P) 
by the last sentence of \S\ref{subsec-remind}. 
\end{example}

\begin{proposition}\label{propequalityP}
Assume that $(V,\rho)\in \Rep(B)$ satisfies the property (P). Then the inclusion \eqref{eqinclu} is an equality, i.e. \[H^0(\GZip^{\mu},\mathcal{V}(\rho))= \bigoplus_{\substack{i\in -(q-1)\NN, \\ (q+1)i\leq (q-1)m}} V_i.\]
\end{proposition}

\begin{proof}
In this case, the element $\rho (u_2) f$ in the equation \eqref{eqf1} has a nonzero projection onto $V_m$ by Proposition \ref{propPequiv}(iii). Thus if $f\in H^0(\GZip^\mu,\mathcal{V}(\rho))_i$, then we must have $m\geq \frac{(q+1)i}{q-1}$. This shows that \eqref{eqinclu} is an equality.
\end{proof}

\section{Category of automorphic vector bundles on $\GZip^\mu$}

\subsection{The category $\VB_P(\GZip^\mu)$}
Recall the functor $\Rep(P)\to \VB_P(\GZip^\mu)$ (\S \ref{sec-VBGzip}). 
This functor is not fully faithful 
even after restricting to the full subcategory $\Rep(L)\subset \Rep(P)$ (see \S\ref{sec-Lreps}). 
Indeed, consider the following example.
 
\begin{example}
Assume that $P$ is defined over $\mathbb{F}_q$. 
Let $\mathbf{1}\in \Rep(L)$ be the trivial $L$-representation, and $(V,\rho) \in \Rep(L)$. 
Then $\Hom_{\Rep(L)}(\mathbf{1},V)=V^L$, whereas we have 
 \begin{equation}
     \Hom_{\VB(\GZip^\mu)}(\mathcal{V}(\mathbf{1}),\mathcal{V}(\rho)) = H^0(\GZip^\mu,\mathcal{V}(\rho))=V^{L(\FF_q)}\cap V_{\geq 0}^{\Delta^P} 
\end{equation}
by Corollary \ref{cor-Fq-Levi}. 
\end{example}
 
To overcome the problem, we introduce  $L_{\varphi}$-modules with additional structures. 

\begin{definition}\label{defNmod}
An $L_{\varphi}$-module 
with $\Delta^P$-monodromy is a pair  $((\tau,V),\Ncal)$ where $\tau \colon L_{\varphi} \to \GL_k(V)$ is a finite-dimensional representation of $L_{\varphi}$ with a decomposition 
$V=\bigoplus_{\nu \in X^*(T)} V_{\nu}$ as $k$-vector spaces and 
$\Ncal = \{N_{\alpha'}^{(j)}\}_{\alpha\in \Delta^P,\, \alpha' \in \Xi_{\alpha},\, j \in \ZZ}$ is a set of $k$-linear endmorphisms of $V$ 
such that 
$N_{\alpha'}^{(j)}(V_{\nu}) \subset V_{\nu+j\alpha'}$, 
$N_{\alpha'}^{(0)}=\id$ and $N_{\alpha'}^{(j)}=0$ for $j <0$. 

Morphisms are given as follows: Let $((\tau,V),\Ncal)$ and $((\tau',V'),\Ncal')$ be two $L_{\varphi}$-modules 
with $\Delta^P$-monodromy. Then a morphism $((\tau,V),\Ncal)\to ((\tau',V'),\Ncal')$ is a $k$-linear map $f \colon V\to V'$ which satisfies:
 \begin{enumerate}[(1)]
\item $f$ is an $L_{\varphi}$-equivariant morphism.
\item 
For $\alpha \in \Delta^P$, $[\mathbf{j}]\in \ZZ^{m_{\alpha}}/(\ZZ^{m_{\alpha}})_{r_{\alpha}}$ 
and $\chi \in X^*(T)$ such that $[\mathbf{j}]\cdot \mathbf{r}_{\alpha}>\delta_{\alpha}(\chi)$, we have 
\[
 \sum_{\mathbf{j} \in [\mathbf{j}]} \sum_{\mathbf{j}' \in \ZZ^{m_{\alpha}}} (-1)^{\sum_{i=1}^{m_{\alpha}}j'_i}
  \pr_{\chi} \left( N_{\alpha_1}'^{(j'_1)}  \cdots  N_{\alpha_{m_{\alpha}}}'^{(j'_{m_{\alpha}})}   f   N_{\alpha_{m_{\alpha}}}^{(j_{m_{\alpha}} -j'_{m_{\alpha}})}  \cdots N_{\alpha_1}^{(j_1 -j'_1)} \right) =0 , 
\]
where $\pr_{\chi}$ denotes the projection 
\[
 \pr_{\chi} \colon \Hom (V,V') \simeq 
 \bigoplus_{\nu,\nu' \in X^*(T)} \Hom (V_{\nu},V'_{\nu'}) 
 \to 
 \bigoplus_{\nu \in X^*(T)} \Hom (V_{\nu},V'_{\nu+\chi}). 
\] 
\end{enumerate}
We denote by $L_{\varphi}\mathchar`-\mathrm{MN}_{\Delta^P}$ the category of 
$L_{\varphi}$-modules with $\Delta^P$-monodromy. 
\end{definition}

\begin{rmk}\label{rmkNcond}
The condition (2) in Definition \ref{defNmod} 
means that $f$ is compatibile with $\Ncal$ and $\Ncal'$ in some sense. 
Assume that $P$ is deffined over $\FF_q$. 
Then the condition (2) in Definition \ref{defNmod} 
is simplified as follows: 
For $\alpha \in \Delta^P$, $\chi \in X^*(T)$ and 
$j\in \NN$ such that $j r_{\alpha,1}>\delta_{\alpha}(\chi)$, we have 
\[
 \pr_{\chi} \left( 
 \sum_{0 \leq j' \leq j} (-1)^{j'} 
  N_{-\alpha}'^{(j')}  f N_{-\alpha}^{(j -j')} \right) =0 . 
\] 
The morphism $N_{-\alpha}^{(j)}$ is an analogue of 
$N^j/j!$ for a monodromy operator $N$ in characteristic zero. In this sense 
\[
 f \mapsto  \sum_{0 \leq j' \leq j} (-1)^{j'} 
  N_{-\alpha}'^{(j')}  f N_{-\alpha}^{(j -j')} 
\]
is an analogue of $j$-th iterate of 
\[
 f \mapsto fN -N' f 
\]
divided by $j!$ 
for monodromy operators $N$ and $N'$ 
in characteristic zero. 
\end{rmk}

We have the functor  
 \begin{equation}
 F_{\mathrm{MN}} \colon \Rep(P)\to L_{\varphi}\mathchar`-\mathrm{MN}_{\Delta^P};\ 
 (V,\rho) \mapsto  \left( (V,\rho|_{L_{\varphi}}),\{E_{\alpha'}^{(j)}\}_{\alpha\in \Delta^P,\, \alpha' \in \Xi_{\alpha},\, j \in \ZZ} \right) 
\end{equation}
where we equip $V$ with the natural $T$-weight decomposition $V= \bigoplus_{\nu}V_\nu$. 

\begin{definition}\label{defadmNmod}
An $L_{\varphi}$-module with $\Delta^P$-monodromy 
is called admissible if it is in the essential image of $F_{\mathrm{MN}}$. 
We denote by $L_{\varphi}\mathchar`-\mathrm{MN}_{\Delta^P}^{\mathrm{adm}}$ the category of admissible $L_{\varphi}$-modules with $\Delta^P$-monodromy.
\end{definition}

\begin{theorem}\label{equivNmod}
The functor $\mathcal{V}\colon \Rep(P)\to \VB(\GZip^\mu)$ factors through the functor $F_{\mathrm{MN}} \colon \Rep(P)\to L_{\varphi}\mathchar`-\mathrm{MN}_{\Delta^P}^{\mathrm{adm}}$ and induces an equivalence of categories
\begin{equation}
 L_{\varphi}\mathchar`-\mathrm{MN}_{\Delta^P}^{\mathrm{adm}} \longrightarrow \VB_P(\GZip^\mu).
\end{equation}
\end{theorem}
\begin{proof}
For two $P$-representations $(V,\rho)$ and $(V',\rho')$, one has 
\begin{align}
    \Hom_{\VB(\GZip^\mu)}(\mathcal{V}(\rho),\mathcal{V}(\rho'))&=\Hom_{\VB(\GZip^\mu)}(\mathcal{V}(\mathbf{1}),\mathcal{V}(\rho)^\vee\otimes \mathcal{V}(\rho'))\\
    &=\Hom_{\VB(\GZip^\mu)}(\mathcal{V}(\mathbf{1}),\mathcal{V}(\rho^\vee\otimes \rho'))\\
    &=H^0(\GZip^\mu,\mathcal{V}(\rho^\vee\otimes \rho'))\\
    &=(V^\vee\otimes V')^{L_{\varphi}}\cap 
     \bigcap_{\alpha \in \Delta^P} 
 \bigoplus_{[\nu] \in X^*(T)/\Lambda_{\Xi_{\alpha},\mathbf{r}_{\alpha}}} 
 \fil_{\delta_{\alpha}}^{\Xi_{\alpha},\mathbf{a}_{\alpha},\mathbf{r}_{\alpha}} (V^\vee\otimes V')_{[\nu]} 
\end{align}
where we used Theorem \ref{thm-main} in the last line. We can see from the definition that this space coincides with the space of homomorphisms $F_{\mathrm{MN}}(V,\rho)\to F_{\mathrm{MN}}(V',\rho')$ using that 
the action of $u_{\alpha'}(x)$ on $V^\vee\otimes V'$ 
is given by 
$f \mapsto \rho'(u_{\alpha'}(x)) \circ f \circ \rho(u_{\alpha'}(-x))$ 
for $\alpha' \in \Xi_{\alpha}$. 
\end{proof}

Let $S_K$ denote the good reduction special fiber of a Hodge-type Shimura variety, with the same notations and assumptions as in \S \ref{subsec-Shimura}. Recall that there is a functor $\mathcal{V}\colon \Rep(P)\to \VB(S_K)$ (see \eqref{equ-Shimura-functor}), which induces functors
\begin{equation}
\xymatrix@1{
\Rep(P) \ar[r]^-{\mathcal{V}} & \VB_P(\GZip^\mu) \ar[r]^-{\zeta^*} & \VB_P(S_K)
}
\end{equation}
where $\VB_P(S_K)$ also denotes the essential image of $\Rep(P)$ in $\VB(S_K)$.
We obtain the following corollary in the context of Shimura varieties. 

\begin{corollary}\label{cor-factorize-Shimura}
The functor $\mathcal{V} \colon \Rep(P)\to \VB_P(S_K)$ factors as 
\begin{equation}
\xymatrix@1@M=5pt{
\Rep(P) \ar[r]^-{F_{\mathrm{MN}}} &  L_{\varphi}\mathchar`-\mathrm{MN}_{\Delta^P}^{\mathrm{adm}} \ar[r]^-{\zeta^*} & \VB_P(S_K).
}
\end{equation}
\end{corollary}

\subsection{The category $\VB_L(\GZip^\mu)$}

We assume that $P$ is defined over $\FF_q$. Hence, in what follows, we have $L_\varphi=L(\FF_q)$.

\begin{definition}\label{defVBL}
Let $\VB_L(\GZip^\mu)$ denote the full subcategory of $\VB(\GZip^\mu)$ which is equal to the essential image of the functor $\Rep(L)\to \VB(\GZip^\mu)$. We call it the category of $L$-vector bundles on $\GZip^\mu$.
\end{definition}

For example, the automorphic vector bundles $(\mathcal{V} (\lambda))_{\lambda \in X^*(T)}$ (see \S \ref{sec-Lreps}) lie in the subcategory of $L$-vector bundles on $\GZip^\mu$. 

\begin{definition}\label{deffilmod}
A $\Delta^P$-filtered $L_{\varphi}$-module is a pair $((\tau,V),\Fcal)$ where $\tau \colon L_{\varphi} \to \GL_k(V)$ is a finite-dimensional representation of $L_{\varphi}$ and $\Fcal = \{V_{\geq \bullet}^\alpha\}_{\alpha\in \Delta^P}$ is a set of filtrations on $V$.  
Here, $V_{\geq \bullet}^\alpha$ denotes a descending filtration $(V_{\geq r}^\alpha)_{r\in \RR}$.

Morphisms are given as follows. Let $((\tau,V),\Fcal)$ and $((\tau',V'),\Fcal')$ be two $\Delta^P$-filtered $L_{\varphi}$-modules. Then a morphism  $((\tau,V),\Fcal)\to ((\tau',V'),\Fcal')$ is a $k$-linear map $f \colon V\to V'$ which satisfies:
 \begin{enumerate}[(1)]
\item $f$ is an $L_{\varphi}$-equivariant morphism.
\item For each $\alpha \in \Delta^P$, 
the map $f$ is compatible with the filtrations $V_{\geq \bullet}^\alpha$ and $V'^\alpha_{\geq \bullet}$ in the sense that 
$f(V^{\alpha}_{\geq r}) \subset V'^{\alpha}_{\geq r}$ 
for any $r \in \RR$. 
\end{enumerate}
We denote by $L_{\varphi}\mathchar`-\mathrm{MF}_{\Delta^P}^{\mathrm{adm}}$ the category of $\Delta^P$-filtered $L_{\varphi}$-modules. 
\end{definition}

Let $((\tau,V),\Ncal) \in L_{\varphi}\mathchar`-\mathrm{MN}_{\Delta^P}$. 
For $\alpha\in \Delta^P$, define the $\alpha$-filtration $(V^{\alpha}_{\geq \bullet})$ of $V$ as follows: Let $V = \bigoplus_{\nu}V_\nu$ be the weight decomposition of $V$. For all $r\in \RR$, let $V^{\alpha}_{\geq r}$ be the direct sum of $V_\nu$ for all $\nu$ satisfying $\langle \nu,\delta_{\alpha} \rangle \geq r$. We call $V^{\alpha}_{\geq \bullet}$ the $\alpha$-filtration of $V$. 
Thus we have a functor 
$L_{\varphi}\mathchar`-\mathrm{MN}_{\Delta^P} \to L_{\varphi}\mathchar`-\mathrm{MF}_{\Delta^P}$. 
Taking composition, we obtain 
 \begin{equation}\label{equ-functor-modfil}
 F_{\mathrm{MF}} \colon \Rep(L) \to \Rep (P) \xrightarrow{F_{\mathrm{MN}}} 
 L_{\varphi}\mathchar`-\mathrm{MN}_{\Delta^P} \to 
 L_{\varphi}\mathchar`-\mathrm{MF}_{\Delta^P}. 
\end{equation}

\begin{definition}\label{defadmfilmod}
A $\Delta^P$-filtered $L_{\varphi}$-module is called admissible if it is in the essential image of $F_{\mathrm{MF}}$. 
We denote by $L_{\varphi}\mathchar`-\mathrm{MF}_{\Delta^P}^{\mathrm{adm}}$ the category of admissible $\Delta^P$-filtered $L_{\varphi}$-modules. 
\end{definition}

\begin{theorem}\label{fullfaith}
The functor $\mathcal{V}\colon \Rep(L)\to \VB(\GZip^\mu)$ factors through the functor $F_{\mathrm{MF}} \colon \Rep(L)\to L_{\varphi}\mathchar`-\mathrm{MF}_{\Delta^P}^{\mathrm{adm}}$ and induces an equivalence of categories
\begin{equation}
 L_{\varphi}\mathchar`-\mathrm{MF}_{\Delta^P}^{\mathrm{adm}} \longrightarrow \VB_L(\GZip^\mu).
\end{equation}
\end{theorem}
\begin{proof}
By Theorem \ref{equivNmod}, it suffices to show 
\begin{align}
 \Hom_{L_{\varphi}\mathchar`-\mathrm{MN}_{\Delta^P}}( F_{\mathrm{MN}}(\rho),F_{\mathrm{MN}}(\rho') )=
 \Hom_{L_{\varphi}\mathchar`-\mathrm{MF}_{\Delta^P}}( F_{\mathrm{MF}}(\rho),F_{\mathrm{MF}}(\rho') ) 
\end{align}
for $(V,\rho), (V',\rho') \in \Rep(L)$. 
This follows from Remark \ref{rmkNcond} and the definitions of morphisms in $L_{\varphi}\mathchar`-\mathrm{MN}_{\Delta^P}$ and $L_{\varphi}\mathchar`-\mathrm{MF}_{\Delta^P}$. 
\end{proof}

\section{Examples}\label{sec-examples}

\subsection{The algebras $R_I$ and $R_{\Delta}$}

Fix a connected reductive group $G$ over $\FF_q$, a cocharacter $\mu \colon \GG_{\mathrm{m},k}\to G_k$, and a frame $(B,T,z)$ for $\Zcal_\mu$ (\S\ref{sec-frames}). 
For $\lambda\in X^*_{+}(T)$, denote by $V_{\Delta}(\lambda)$ the $G$-representation $\Ind_B^G(\lambda)$. 
We add a subscript $\Delta$ to avoid confusion with $V_I (\lambda)=\Ind_{B_L}^L(\lambda)$ for $\lambda\in X^*_{+,I}(T)$ (see \S\ref{sec-Lreps}). Let $\mathcal{V}_{\Delta}(\lambda)$ be the vector bundle on $\GZip^\mu$ attached to $V_{\Delta}(\lambda)$. 
We put 
\begin{equation}\label{RandRG}
    R_I = \bigoplus_{\lambda\in X^*_{+,I}(T)} H^0(\GZip^\mu,\mathcal{V}_I (\lambda))\quad \textrm{ and }\quad R_{\Delta} = \bigoplus_{\lambda\in X^*_+(T)} H^0(\GZip^\mu,\mathcal{V}_{\Delta}(\lambda)).
\end{equation}
By \eqref{Vlambdamap}, the $k$-vector spaces $R_I$ and $R_{\Delta}$ have a natural structure of $k$-algebra. They capture information about all $\mathcal{V}_I (\lambda)$ and $\mathcal{V}_{\Delta}(\lambda)$ at once. 

\begin{rmk}
In general, we do not know whether $R_I$ and $R_{\Delta}$ are finite-type algebras, but we conjecture it is the case. The algebra $R_I$ was studied in \cite{Koskivirta-automforms-GZip}. In the case of $G=\Sp(4)$ with a cocharacter $\mu$ whose centralizer Levi subgroup is isomorphic to $\GL_{2}$, we showed that $R_I$ is a polynomial algebra in three indeterminates (\cite[Theorem 5.4.1]{Koskivirta-automforms-GZip}).
\end{rmk}

In this first example, we examine $R_{\Delta}$ in the case of $G=\SL_{2,\FF_q}$ with the zip datum explained in \S \ref{SL2zip-sec}. In this case, the algebra $R_I$ is very simple, it is a polynomial algebra in one indeterminate, generated by the classical Hasse invariant. 
Let $n \in \NN$. 
The representation $V_{\Delta}(\chi_n)$ identifies with $\Sym^n(\Std)$. 
The weights of $V_{\Delta}(\chi_n)$ are $\{-n+2i \mid 0\leq i\leq n\}$. By Example \ref{egIndP} and Proposition \ref{propequalityP}, we have 
\begin{equation}\label{equH0}
 H^0(\GZip^\mu,\mathcal{V}_{\Delta}(\chi_n))= \bigoplus_{\substack{i\in -(q-1)\NN, \\ (q+1)i\leq -(q-1)n}} V_{\Delta}(\chi_n)_i
\end{equation}
for all $n\geq 0$. 
Let $x,y$ be indeterminates. Let $\SL_2$ act on $k[x,y]$ by 
\[
 \begin{pmatrix}
    a&b\\c&d
 \end{pmatrix} \cdot P=P(ax+cy,bx+dy). 
\] 
Then $V_{\Delta}(\chi_n)=\Sym^n(\Std)$ is the subrepresentation of $k[x,y]$ spanned by homogeneous polynomials in $x,y$ of degree $n$. The highest weight vector is $x^n$. By \eqref{equH0}, we have
\begin{equation}
H^0(\GZip^\mu,\mathcal{V}_{\Delta}(\chi_n))=\Span_k \left( x^j y^{n-j} \relmiddle| j\geq 0,\ q-1|n-2j, \ (q+1)j\leq n \right). 
\end{equation}
Similarly, $R_{\Delta}$ is the subalgebra of $k[x,y]$ generated by $x^j y^{n-j}$ for all $0\leq j \leq n$ with $q-1|n-2j$ and $(q+1)j\leq n$.

\begin{proposition} \label{prop-RG}
The algebra $R_{\Delta}$ is generated by $y^{q-1}$ and $xy^q$. In particular, it is a polynomial algebra in two indeterminates.
\end{proposition}

\begin{proof}
It is clear that $y^{q-1}$ and $xy^q$ are elements of $R_{\Delta}$. Let $n\geq 0$ and $0\leq j \leq n$ such that $x^jy^{n-j}\in R_{\Delta}$. We can write $x^j y^{n-j} = (xy^q)^j y^{n-(q+1)j}$. Note that $n\geq (q+1)j$ and $q-1$ divides $n-(q+1)j=n-2j-(q-1)j$. It follows that $x^jy^{n-j}$ lies in the subalgebra of $k[x,y]$ generated by $y^{q-1}$ and $xy^q$.
\end{proof}

We give an interpretation of these sections. In the case of $G=\SL_{2,\FF_q}$, recall that for an $\FF_q$-scheme $S$, the groupoid $\GZip^\mu(S)$  consists of tuples $\underline{\Hcal}=(\Hcal,\omega,F,V)$ where
\begin{enumerate}[(1)]
    \item $\Hcal$ is a locally free $\Ocal_S$-module of rank $2$ with a trivialization $\det(\Hcal)\simeq \Ocal_S$,
    \item $\omega\subset \Hcal$ is a locally free $\Ocal_S$-submodule of rank $1$ such that $\Hcal/\omega$ is locally free,
    \item $F \colon \Hcal^{(q)}\to \Hcal$ and $V \colon \Hcal \to \Hcal^{(q)}$ are $\Ocal_S$-linear maps satisfying the conditions $\Ker(F)=\Im(V)=\omega^{(q)}$ and $\Ker(V)=\Im(F)$.
\end{enumerate}
Consider the flag space $\Fcal_G$ over $\GZip^\mu$ parametrizing pairs $(\underline{\Hcal},\mathcal{L})$ with $\mathcal{L}\subset \Hcal$ a  locally free $\Ocal_S$-submodule of rank $1$ such that $\Hcal/\mathcal{L}$ is locally free. The natural projection map $\pi_G \colon \Fcal_G\to \GZip^\mu$ is a $\PP^1$-fibration. For $n\in \ZZ$, the push-forward $\pi_{G,*}(\mathcal{L}^{-n})$ coincides with the vector bundle $\mathcal{V}_{\Delta}(\chi_n)$. Consider the map
\begin{equation}
\mathcal{L} \subset \Hcal \xrightarrow{V} \Hcal^{(q)}\to (\Hcal/\mathcal{L})^{(q)}\simeq \mathcal{L}^{-q},
\end{equation}
where we used that $\Hcal/\mathcal{L}\simeq \mathcal{L}^{-1}$ by the trivialization $\det(\Hcal)\simeq \Ocal_S$. We obtain a section of $\mathcal{L}^{-(q+1)}$. It corresponds to the element $x y^q$ in Proposition \ref{prop-RG}. On the other hand, the classical Hasse invariant $\mathit{Ha} \in H^0(S,\omega^{q-1})$ is given by the map $V \colon \omega \to \omega^{(q)}\simeq \omega^q$. By sending  $\mathit{Ha}$ under the morphism 
\begin{equation}
\omega \subset \Hcal \to \Hcal/\mathcal{L}\simeq \mathcal{L}^{-1},
\end{equation}
we obtain a section of $\mathcal{L}^{-(q-1)}$. This section corresponds to $y^{q-1}$ in Proposition \ref{prop-RG}.

\subsection{Example on $L$-semisimplification}\label{counter-Lss}
We give an example which shows that Proposition \ref{Lssreg} does not hold in general without the assumption that $P$ is defined over $\FF_q$. Let $G=\res_{\FF_{q^2}/\FF_q} \SL_{2,\FF_{q^2}}$ 
and 
\[
 \mu \colon \GG_{\mathrm{m},k} \to G_k \simeq \SL_{2,k} \times \SL_{2,k};\ z \mapsto 
 \left( 
 \begin{pmatrix}
 z & 0 \\ 0 & z^{-1} 
 \end{pmatrix}, 
 \begin{pmatrix}
 1 & 0 \\ 0 & 1 
 \end{pmatrix} 
 \right). 
\] 
Let $\Zcal_{\mu}=(G,P,L,Q,M,\varphi)$ be the associated zip datum. 
We have 
$P=B_2 \times \SL_{2,k}$, $L=T_2 \times \SL_{2,k}$, $Q=\SL_{2,k} \times B_2^+$ and $M=\SL_{2,k} \times T_2$. 
We take $(B,T)=(B_2 \times B_2,T_2 \times T_2)$ as a Borel pair and 
take a frame as in Lemma \ref{lem-framemu}. 
Then $\Delta^P$ consists of one root $\alpha =\chi_2 \boxtimes \chi_0$. 
We have 
\[
 L_{\varphi} = \left\{ 
 \left( 
 \begin{pmatrix} x & 0 \\ 0 & x^{-1} \end{pmatrix}, 
 \begin{pmatrix} x^q & y \\ 0 & x^{-q} \end{pmatrix}
 \right) \in 
 L \relmiddle| 
 x \in \FF_{q^2}^{\times}, y^q=0 
 \right\} . 
\]
We have 
\[
 \delta_{\alpha}=\frac{-\alpha^{\vee}-q\sigma(\alpha^{\vee})}{q^2 -1}, \quad 
 \mathbf{r}_{\alpha}=\left( \frac{q^2+1}{q^2-1}, \frac{-(q^2+1)}{q(q^2-1)} \right), \quad 
 (\ZZ^2)_{\mathbf{r}_{\alpha}}=
 \{(n_1,n_2) \in \ZZ^2 \mid qn_1 =n_2 \}. 
\]
We define $\rho \colon P\to \GL(V)$ by 
\[
 \left( \Sym^{q^2-1}(\Std) \otimes \chi_{q^2-1} \right) \boxtimes 
 \Sym^{q^2-1}(\Std^{(q)}). 
\]
We write $(V',\rho')$ for 
$(V^{L\textrm{-ss}},\rho^{L\textrm{-ss}})$. 
Then we have 
$V^{L_\varphi}=V$ and $V'^{L_\varphi}=V'$. 
We put $\nu=\chi_0 \boxtimes \chi_{-q(q^2-3)}$. 
We have 
\[
 V_{[\nu]} =
 V_{\nu} \oplus 
 V_{\nu +\alpha-q \sigma(\alpha)}. 
\]
We parametrize elements $[\mathbf{j}]\in \ZZ^2/(\ZZ^2)_{\mathbf{r}_\alpha}$ by classes $[(0,j)]$ with $j\in \ZZ$. 
Using this notation, we have 
\begin{align*}
 &\fil_{\delta_{\alpha}}^{\Xi_{\alpha},\mathbf{a}_{\alpha},\mathbf{r}_{\alpha}} V_{[\nu]} = 
 \bigcap_{j \in \ZZ} 
 \bigcap_{\substack{\chi \in [\nu +j \sigma(\alpha)],\\ j r_{\alpha,2} >\delta_{\alpha}(\chi)}} 
 \Ker \left( \sum_{j_1 \in \ZZ} 
 \pr_{\chi} \circ  E_{-\alpha}^{(j_1)} \circ E_{\sigma(\alpha)}^{(j+qj_1)} \colon 
 V_{[\nu]} \to V_{\chi}
 \right) 
\end{align*}
because $(-1)^{j_1} (-1)^{j+q j_1}=(-1)^j \in k$. 
We have $V_{\chi}\neq 0$ if and only if 
$\chi=\nu+i_1 \alpha +q i_2 \sigma(\alpha)$ for $0 \leq i_1 \leq q^2-1$ and $-1 \leq i_2 \leq q^2-2$. 
For $\chi=\nu+i_1 \alpha +q i_2 \sigma(\alpha)$, 
the conditions 
$\chi \in [\nu +j \sigma(\alpha)]$ and 
$j r_{\alpha,2} >\delta_{\alpha}(\chi)$ 
hold if and only if 
$j=q(i_1+i_2)$ and 
$i_2 -i_1 >q^2-2-2/(q^2-1)$. 
Hence 
\[ 
 \chi \in [\nu +j \sigma(\alpha)],\  
 j r_{\alpha,2} >\delta_{\alpha}(\chi),\ 
 V_{\chi}\neq 0 
 \Longleftrightarrow 
 \chi=\nu+q(q^2-2)\sigma(\alpha),\ 
 j=q(q^2-2). 
\]
We put 
$\chi_0=\nu+q(q^2-2)\sigma(\alpha)$ and $j_0=q(q^2-2)$. 
Then we have 
\begin{align*}
 \fil_{\delta_{\alpha}}^{\Xi_{\alpha},\mathbf{a}_{\alpha},\mathbf{r}_{\alpha}} V_{[\nu]} 
 &= 
 \Ker \left(  \pr_{\chi_0} 
 \circ \left( 
  E_{\sigma(\alpha)}^{(j_0)} +
  E_{-\alpha}^{(1)} \circ E_{\sigma(\alpha)}^{(j_0 +q)} 
 \right) \colon 
 V_{[\nu]} \to V_{\chi_0} \right) \\
 &= \Bigl\{ (v_1,v_2) \in  V_{\nu} \oplus 
 V_{\nu+\alpha -q \sigma(\alpha)} 
 \Bigm| E_{\sigma(\alpha)}^{(j_0)} (v_1) +
  (E_{-\alpha}^{(1)} \circ E_{\sigma(\alpha)}^{(j_0 +q)})(v_2) =0 \Bigr\}. 
\end{align*} 
We note that 
\[
 E_{\sigma(\alpha)}^{(j_0)} \colon 
 V_{\nu} \to V_{\chi_0}, \quad 
 E_{-\alpha}^{(1)} \circ E_{\sigma(\alpha)}^{(j_0 +q)}
 \colon 
 V_{\nu+\alpha -q \sigma(\alpha)} \to V_{\chi_0}
\]
are isomorphisms. 
In the same way, we have 
\[
 \fil_{\delta_{\alpha}}^{\Xi_{\alpha},\mathbf{a}_{\alpha},\mathbf{r}_{\alpha}} V'_{[\nu]} 
 = 
 \Ker \left(  \pr_{\chi_0} 
 \circ 
  E_{\sigma(\alpha)}^{(j_0)} 
 \colon 
 V'_{[\nu]} \to V'_{\chi_0} \right) 
  =V'_{\nu+\alpha -q \sigma(\alpha)} 
\]
using $E_{-\alpha}^{(1)}=0$ for $(V',\rho')$. 
Hence  
$\fil_{\delta_{\alpha}}^{\Xi_{\alpha},\mathbf{a}_{\alpha},\mathbf{r}_{\alpha}} V_{[\nu]} \not\subset \fil_{\delta_{\alpha}}^{\Xi_{\alpha},\mathbf{a}_{\alpha},\mathbf{r}_{\alpha}} V'_{[\nu]}$. 
Therefore we have 
$H^0(\GZip^\mu,\mathcal{V}(\rho)) \not\subset 
H^0(\GZip^\mu,\mathcal{V}(\rho'))$. 

\subsection{The case of the unitary group $\U (2,1)$ with $p$ inert}
In this section, we examine an example that arises in the study of Picard surfaces. These are Shimura varieties of PEL-type (in particular, of Hodge-type) attached to unitary groups $\mathbf{G}$ over $\QQ$ with respect to some totally imaginary quadratic extension $\mathbf{E}/\QQ$. 
We impose that $\mathbf{G}_\RR \simeq \GU(2,1)$. We choose a rational prime $p$ that is inert in $\mathbf{E}$ and consider the attached zip datum $(G,P,Q,L,M,\varphi)$. 
Since $p$ is inert, the parabolic $P$ is not defined over $\FF_p$. We study the space $H^0(\GZip^\mu,\mathcal{V}_I (\lambda))$. To simplify, we will work with a unitary group $\U$, instead of a group of unitary similitudes $\GU$. The case of $\GU$ is very similar. 

Let $(V,\psi)$ be a $3$-dimensional vector space over $\FF_{q^2}$ endowed with a non-degenerate hermitian form $\psi \colon V\times V\to \FF_{q^2}$ (in the context of Shimura varieties, take $q=p$). Write $\Gal(\FF_{q^2}/\FF_q)=\{\id,\sigma\}$.  We take a basis $\Bcal=(v_1,v_2,v_3)$ of $V$ where $\psi$ is given by the matrix 
\[
 J= \begin{pmatrix}
&&1\\&1&\\1&&\end{pmatrix}. 
\]
We define a reductive group $G$ by
\begin{equation}
G(R) = \{f\in \GL_{\FF_{q^2}}(V\otimes_{\FF_q} R) \mid  \psi_R(f(x),f(y))=\psi_R(x,y), \ \forall x,y\in V\otimes_{\FF_q} R \}
\end{equation}
for any $\FF_q$-alegebra $R$. One has an identification $G_{\FF_{q^2}}\simeq \GL(V)$, given as follows: For any $\FF_{q^2}$-algebra $R$, we have an $\FF_{q^2}$-algebra isomorphism $\FF_{q^2}\otimes_{\FF_q} R\to R\times R$, $a\otimes x\mapsto (ax,\sigma(a)x)$. By tensoring with $V$, we obtain an isomorphism $V\otimes_{\FF_q}R\to (V\otimes_{\FF_{q^2}}R)\oplus (V\otimes_{\FF_{q^2}}R)$. Then any element of $G(R)$ stabilizes this decomposition, and is entirely determined by its restriction to the first summand. This yields an isomorphism as claimed. Using the basis $\Bcal$, we identify $G_{\FF_{q^2}}$ with $\GL_{3,\FF_{q^2}}$. The action of $\sigma$ on the set $\GL_3(k)$ is given as follows: $\sigma\cdot A = J \sigma({}^t \!A)^{-1}J$. Let $T$ denote the maximal diagonal torus and $B$ the lower-triangular Borel subgroup of $G_k$. Note that by our choice of the basis $\Bcal$, the groups $B$ and $T$ are defined over $\FF_q$. Identify $X^*(T)=\ZZ^3$ such that $(k_1,k_2,k_3)\in \ZZ^3$ corresponds to the character $\diag(x_1,x_2,x_3)\mapsto \prod_{i=1}^3 x_i^{k_i}$. The simple roots are $\Delta=\{e_1-e_2, e_2-e_3 \}$, where $(e_1,e_2,e_3)$ is the canonical basis of $\ZZ^3$. 

Define a cocharacter $\mu  \colon  \GG_{\mathrm{m},k}\to G_{k}$ such that  $\mu$ is given by $x\mapsto \diag(x,x,1)$ via the identification $G_{k}\simeq \GL_{3,k}$. 
Let $\Zcal_{\mu}=(G,P,L,Q,M,\varphi)$ be the associated zip datum. Note that $P$ is not defined over $\FF_q$. One has $I=\{e_1-e_2\}$ and  $\Delta^P=\{\alpha\}$ with $\alpha=e_2-e_3$. 

\begin{lemma}\label{lemma-mu-ord-Hasse-invariant-GU21}
Let $H$ be the function on $G_k$ defined by 
\[
 H\left( (x_{i,j})_{1\leq i,j\leq 3}\right) = x_{1,1}^q \Delta_1 - x_{2,1}^q \Delta_2 
 \quad \textrm{with } 
\begin{cases}
\Delta_1=x_{1,1}x_{2,2}-x_{1,2}x_{2,1}, \\
\Delta_2=x_{1,1}x_{2,3}-x_{2,1}x_{1,3}.
\end{cases}
\]
The $\mu$-ordinary stratum $U_\mu\subset G_k$ 
is equal to the complement of 
the vanishing locus of $H$. 
\end{lemma}

\begin{proof}
In this case, there is a unique $E$-orbit of codimension $1$ by the first part of Theorem \ref{thm-E-orb-param}. Furthermore, this $E$-orbit is dense in $G_k \setminus U_\mu$ by the closure relation. Hence, it suffices to show that $H$ does not vanish on $U_\mu$. The group $E$ consists of pairs $(x,y)\in P\times Q$ with
\[
x=\left(\begin{matrix}
a&b&0\\ c&d& 0\\ e&f&g
\end{matrix} \right) \quad \textrm{ and } \quad y=\left(\begin{matrix}
g^q&h&i\\ 0&d^q& b^q\\ 0&c^q&a^q
\end{matrix} \right)^{-1}.
\]
Since $1\in U_\mu$, the open $U_\mu$ consists of elements of the form $xy^{-1}$. We find
\begin{align}
H(xy^{-1}) = (ag^q)^q g^q d^q (ad-bc) - (cg^q)^q g^q b^q (ad-bc) = g^{q^2+q}(ad-bc)^{q+1}.
\end{align}
This expression is nonzero, so the result is proved.
\end{proof}

We have 
\begin{equation}\label{eq-Lvarphi-GU21}
L_\varphi = \left\{ 
\begin{pmatrix}
a&b &\\ &d& \\ &&a^{-q}
\end{pmatrix} \in L \relmiddle| a,d \in \FF_{q^2}^{\times},\ d^{q+1}=1,\ b^q=0 
\right\}. 
\end{equation}
The endomorphism $\wp_* \colon X_*(T)_{\RR} \to X_*(T)_{\RR}$ is given by the matrix
\begin{equation}
\wp_*=\left( 
\begin{matrix}
1&&q \\ &1+q& \\ q&&1
\end{matrix}
\right).
\end{equation}
Hence it follows that $\delta_\alpha = \wp_*^{-1}(\alpha^\vee)=\frac{1}{q^2-1}(-q,q-1,1)$. We have $m_\alpha=2$, $\mathbf{a}_\alpha=(-1,-1)$, $\Xi_\alpha=(-\alpha,\sigma(\alpha))$, and
\[
 \mathbf{r}_\alpha=\left(\frac{q^2-q+1}{q^2-1},\frac{-q^2+q-1}{q(q^2-1)} \right), \quad (\ZZ^2)_{\mathbf{r}_\alpha}=\{(n_1,n_2)\in \ZZ^2 \mid  qn_1=n_2\}.
\]
The group $\Lambda_{\Xi_{\alpha},\mathbf{r}_{\alpha}}$ is
\[\Lambda_{\Xi_{\alpha},\mathbf{r}_{\alpha}}= \ZZ(q,-(q+1),1). \]

Let $\lambda=(\lambda_1,\lambda_2,\lambda_3)$ be an $L$-dominant character (i.e. $\lambda_1\geq \lambda_2$), and consider the $L$-representation $V_I (\lambda)$. 
We simply write $V$ for $V_I (\lambda)$ sometimes. 
Under the isomorphism 
\[
 \GL_{2}\times \GG_{\mathrm{m}} \to L;\ (A,z) \mapsto 
 \left( \begin{matrix}
 A &\\ & z
 \end{matrix} \right), 
\]
the representation $V$ corresponds to the representation
\[ 
 \deter_{\GL_2}^{\lambda_2} \otimes \Sym^{\lambda_1-\lambda_2}(\Std_{\GL_{2}})\otimes \xi_{\lambda_3} 
\]
where $\xi_r$ is the character of 
$\GL_{2}\times \GG_{\mathrm{m}}$ given by $(A,z)\mapsto z^r$. Hence $V$ is a representation of dimension $\lambda_1-\lambda_2+1$ and it has weights
\[
 \nu_i \colonequals (\lambda_1-i,\lambda_2+i,\lambda_3), \quad 0\leq i \leq \lambda_1-\lambda_2. 
 \]
Note that the difference $\nu_i-\nu_{i'}$ of two weights is never in $\Lambda_{\Xi_{\alpha},\mathbf{r}_{\alpha}}$ unless $i=i'$. Therefore $V_{[\nu]}=V_{\nu}$ for all $\nu\in \ZZ^3$. 
One deduces 
\begin{equation}\label{equ-VlambdaLphi}
V^{L_\varphi}=\bigoplus_{\substack{q|i, \ q+1 | \lambda_2+i, \\ q^2-1| \lambda_1-i-q\lambda_3}} V_{\nu_i}.
\end{equation}
It remains to determine $\fil_{\delta_{\alpha}}^{\Xi_{\alpha},\mathbf{a}_{\alpha},\mathbf{r}_{\alpha}} V_{\nu}$, which is either $0$ or $V_\nu$. We parametrize elements $[\mathbf{j}]\in \ZZ^2/(\ZZ^2)_{\mathbf{r}_\alpha}$ by classes $[(0,j)]$ with $j\in \ZZ$. Then, an element $\mathbf{j}\in [\mathbf{j}]$ can be written as $(0,j)+j_1(1,q)$ with $j_1\in \ZZ$. Using this notation, we obtain
\[
 \fil_{\delta}^{\Xi,\mathbf{a},\mathbf{r}} V_{\nu} = 
 \bigcap_{j\in \ZZ} 
 \bigcap_{\substack{\chi \in [\nu+j \sigma(\alpha)], \\   j r_{\alpha,2} > \delta_\alpha (\chi)}} 
 \Ker \left( \sum_{j_1 \in \ZZ}
 \pr_{\chi} \circ E_{-\alpha}^{(j_1)} \circ  E_{\sigma(\alpha)}^{(j+q j_1)}  \colon V_{\nu} \to 
 V_{\chi} 
 \right) 
\]
because $(-1)^{j_1} (-1)^{j+q j_1}=(-1)^j \in k$. 
We have $E^{(j_1)}_{-\alpha}=0$ unless $j_1=0$ because $\alpha\in \Delta^P$ and $V$ is trivial on $R_{\mathrm{u}}(P)$. Hence in the sum appearing in the above formula, only the case $j_1=0$ contributes. Furthermore, 
$E_{\sigma(\alpha)}^{(j)} (V_{\nu}) \subset 
V_{\nu+j\sigma(\alpha)}$. 
Hence we have 
\[
 \fil_{\delta}^{\Xi,\mathbf{a},\mathbf{r}} V_{\nu} = 
 \bigcap_{j> q\langle \nu, \delta_\alpha \rangle}
 \Ker \left(E_{e_1-e_2}^{(j)}  \colon V_{\nu} \to 
 V_{\nu+j(e_1-e_2)} 
 \right) . 
\]
Take $\nu=\nu_i$ for some $0\leq i \leq \lambda_1-\lambda_2$. We deduce $\fil_{\delta_{\alpha}}^{\Xi_{\alpha},\mathbf{a}_{\alpha},\mathbf{r}_{\alpha}} V_{\nu_i}=V_{\nu_i}$ if and only if for all $j \geq 0$ such that $j > q \langle \nu_i, \delta_\alpha \rangle$, one has $E^{(j)}_{e_1-e_2}(V_{\nu_i})=0$. 
Computing explicitly the representation $V$, one sees that this space is zero if and only if the binomial coefficient $\binom{i}{j}$ is divisible by $p$. In particular, it is never zero for $j=i$. We deduce that
\[\fil_{\delta_{\alpha}}^{\Xi_{\alpha},\mathbf{a}_{\alpha},\mathbf{r}_{\alpha}} V_{\nu_i}=V_{\nu_i} \Longleftrightarrow i \leq q\langle \nu_i,\delta_\alpha \rangle.\]
Furthermore, we find
\begin{equation}\label{equ-scalar-prod-taui}
\langle \nu_i,\delta_\alpha \rangle=\frac{i(2q-1)}{q^2-1}+\frac{1}{q^2-1}(-q\lambda_1+(q-1)\lambda_2+\lambda_3).
\end{equation}
For $\lambda=(\lambda_1,\lambda_2,\lambda_3)\in X_{+,I}^*(T)$, we put 
\[
 F(\lambda) = \frac{q}{q^2-q+1}
 (q\lambda_1-(q-1)\lambda_2-\lambda_3). 
\] 
We deduce:
\begin{proposition}
We have 
\begin{equation}\label{equ-propH0}
 H^0(\GZip^\mu,\mathcal{V}_I (\lambda)) = \bigoplus_{\substack{q|i, \ q+1 | \lambda_2+i, \\ q^2-1| \lambda_1-i-q\lambda_3,\ i \geq F(\lambda)} } V_I (\lambda)_{\nu_i}.
\end{equation}
\end{proposition}

\begin{enumerate}[(1)]
\item \label{U21-item1} For example, take $\lambda=(1+q,1,q)$. Then one sees that $V_I (\lambda)^{L_\varphi}=V_I (\lambda)_{\nu_q}$, where $\nu_q=(1,1+q,q)$. One finds $F(\lambda)=q$, hence $H^0(\GZip^\mu,\mathcal{V}_I (\lambda))=V_I (\lambda)_{\nu_q}$.
\item \label{U21-item2} Similarly, take $\lambda=(1,0,q)$. Then we find $V_I (\lambda)^{L_\varphi}=V_I (\lambda)_{\nu_0}$, where $\nu_0=\lambda=(1,0,q)$. We have $F(\lambda)=0$, hence again $H^0(\GZip^\mu,\mathcal{V}_I (\lambda))=V_I (\lambda)_{\nu_0}$.
\item \label{U21-item3} Take $\lambda=(q+1,q+1,q^2+q)$. Then $V_I (\lambda)$ is a one-dimensional representation of $L$ (i.e. a character), and $V_I (\lambda)^{L_\varphi}=V_I (\lambda)$. Since $F(\lambda) = -\frac{q(q^2-1)}{q^2-q+1}<0$, we have $H^0(\GZip^\mu,\mathcal{V}_I (\lambda))=V_I (\lambda)$. It is spanned by  the $\mu$-ordinary (non-classical) Hasse invariant $H$ given by Lemma \ref{lemma-mu-ord-Hasse-invariant-GU21}, also constructed in \cite{Goldring-Nicole-mu-Hasse} and \cite{Koskivirta-Wedhorn-Hasse}.
\end{enumerate}

Recall the cone $C_{\zip}\subset X_{+,I}^*(T)$ studied in \cite{Koskivirta-automforms-GZip}, \cite{Goldring-Koskivirta-global-sections-compositio}, defined as the set of $\lambda \in X^*(T)$ such that $H^0(\GZip^\mu,\mathcal{V}_I (\lambda))\neq 0$. 
In this example, we deduce that it is the set of $\lambda\in X_{+,I}^*(T)$ such that there exists $0\leq i \leq \lambda_1-\lambda_2$ satisfying the four conditions listed below the direct sum sign of \eqref{equ-propH0}. For a cone $C\subset X^*(T)$, write $\langle C \rangle$ for the saturated cone of $C$, i.e. the set of $\lambda\in X^*(T)$ such that $N\lambda$ lies in $C$ for some positive integer $N$.

\begin{corollary}
We have 
\[
 \langle C_{\zip} \rangle= \left\{  (\lambda_1,\lambda_2,\lambda_3)\in \ZZ^3 
 \mid 
\lambda_1\geq \lambda_2, \ 
(q-1)\lambda_1 +\lambda_2-q\lambda_3\leq 0
  \right\}. 
\]
\end{corollary}
\begin{proof}
Assume that $\lambda\in C_{\zip}$. Then in particular $\lambda_1-\lambda_2 \geq F(\lambda)$, which amounts to $(q-1)\lambda_1 +\lambda_2-q\lambda_3\leq 0$. Conversely, assume that $\lambda \in X^*_{+,I}(T)$ satisfies $\lambda_1-\lambda_2 \geq F(\lambda)$. Then, after changing $\lambda$ to $q(q^2-1)\lambda$, we find that $i=\lambda_1-\lambda_2$ satisfies the four conditions below the direct sum sign of \eqref{equ-propH0}, hence $\lambda\in \langle C_{\zip} \rangle$. This terminates the proof.
\end{proof}

\begin{rmk}
The two sections of weight $(1+q,1,q)$ and $(1,0,q)$ given in \eqref{U21-item1} and \eqref{U21-item2} are partial Hasse invariants (viewing them as section of the stack of zip flags $\GF^\mu$, their vanishing locus is a single flag stratum, see \cite[\S 1.3]{Koskivirta-automforms-GZip} for details). Their weights generate the cone $\langle C_{\Sbt}\rangle$ defined in \cite[Definition 1.7.1]{Koskivirta-automforms-GZip}. The cone $\langle C_{\zip} \rangle$ is not spanned by these weights because $G$ does not satisfy the equivalent conditions of \cite[Lemma 2.3.1]{Koskivirta-automforms-GZip}. We also refer to \cite{Goldring-Imai-Koskivirta-weights} for a general study of the cone $C_{\zip}$ as well as related results.
\end{rmk}

\vspace*{0.8em}

\noindent
{\bf Acknowledgements.} The authors thank the referee for helpful comments and suggestions. 
This work was supported by JSPS KAKENHI Grant Numbers 
18F18311 and 18H01109. 



\newcommand{\etalchar}[1]{$^{#1}$}

\noindent
Naoki Imai\\
Graduate School of Mathematical Sciences, The University of Tokyo, 
3-8-1 Komaba, Meguro-ku, Tokyo, 153-8914, Japan \\
naoki@ms.u-tokyo.ac.jp\\ 

\noindent
Jean-Stefan Koskivirta\\
Department of Mathematics, Faculty of Science, Saitama University, 
255 Shimo-Okubo, Sakura-ku, Saitama City, Saitama 338-8570, Japan \\
jeanstefan.koskivirta@gmail.com

\end{document}